\documentclass[oneside,12pt]{amsart}
\usepackage[utf8]{inputenc}
\usepackage{tikz}
\usepackage{tikz-cd}
\usepackage{enumitem}
\usepackage{stmaryrd}

\usepackage{chngcntr}

\usepackage{mathtools}
\counterwithin{equation}{subsection}

\definecolor{dblue}{rgb}{0.0, 0.0, 0.55}
\definecolor{burgundy}{rgb}{0.5, 0.0, 0.13}
\definecolor{dgreen}{rgb}{0.13, 0.55, 0.13}

\usepackage[hidelinks]{hyperref}
\hypersetup{
  colorlinks   = true, 
  urlcolor     = blue, 
  linkcolor    = blue, 
  citecolor   = blue 
}

\usepackage{amsthm} 
\usepackage{amsmath}
\usepackage{amssymb}
\usepackage{latexsym}
\usepackage{mathtools}
\usepackage{mathdots}
\usepackage{graphicx}
\usepackage{caption}
\usepackage{subcaption}
\usepackage{mathrsfs}
\usepackage{stmaryrd}

\newtheorem{thm}{Theorem}[subsection]
\newtheorem{lem}[thm]{Lemma}
\newtheorem{rmk}[thm]{Remark}
\newtheorem{pro}[thm]{Proposition}
\newtheorem{cor}[thm]{Corollary}

\newtheorem*{thmen}{Theorem}

\newcommand{\gb}[1]{\operatorname{\textbf {\textup{#1}}}}

\title{Langlands functoriality conjecture for $\gb{SO}_{2n}^*$ \\ \hspace{-0,8cm}in positive characteristic}
\author{H\'ector del Castillo}

\address{Pontificia Universidad Católica de Valparaíso, Blanco Viel 596, Cerro Barón, Valparaíso, Chile} 
\email{hector.math@gmail.com}

\begin{document}
\begin{abstract}
In this article we are concerned with the Langlands  functoriality  conjecture.  Cogdell,  Kim,  Piatetski-Shapiro and Shahidi proved functioriality conjecture in  the  case  of  a  globally  generic  cuspidal automorphic  representation  for  the  split  classical groups,  unitary  groups  or  even  quasi-split  special orthogonal  groups  in  characteristic  zero.  Lomelí extends  this  result  to  split  classical  groups  and unitary groups in positive characteristic. Thus, in this article we prove the Langlands  functoriality  conjecture  for  the  even quasi-split  non-split  special  orthogonal  groups  in positive  characteristic i.e. we lift globally   generic   cuspidal   automorphic representations  of   quasi-split  non-split even  special orthogonal groups  to  generic automorphic representations of suitable general linear groups in positive characteristic. As  an  application  of  this  result,  we  prove  the compatibility  of  the  local gamma  factors  and  the unramified Ramanujan conjecture.
\end{abstract}
\maketitle

 \section*{Introduction}The Langlands program plays an important role in Number Theory and Representation Theory. A crucial aspect of this program is the functoriality conjecture:  let $F$ be a global field with ring of ad\`eles and two connected (quasi-split) reductive groups $\gb G$ and $\gb H$ over $F$.  Let
\[\rho \colon {}^LG \to {}^LH,\]
be a given $L$-homomorphism  between the $L$-groups of ${\bf G}$ and ${\bf H}$, respectively. Then, according to this conjecture, for every  cuspidal automorphic representation $\pi=\bigotimes_x'\pi_x$ of $\gb G(\mathbb A_F)$,  there exists an automorphic representation $\Pi=\bigotimes_x' \Pi_x$ of $\gb H(\mathbb A_F)$ such that, at almost all places $x$ where $\pi_x$ is unramified, $\Pi_x$ is unramified and
its Satake parameter corresponds to the image under $\rho$ of the Satake parameter of $\pi_x$. Such representation will be called a lift of $\pi$.   Furthermore the lift process should respect arithmetic information coming from $\gamma$-factors, $L$-functions and $\varepsilon$-factors,  and lead to a local version of functoriality at the ramified places as well.

When $\gb G$ is a classical group, ${}^LG$ has a natural representation into ${}^LH$ for a specific general linear group $\gb H={\bf GL}_N$. That case has been studied by many people using different tools. When $F$ is a number field, two main tools have been used: converse theorem and trace formulas. The former was used by Cogdell, Kim, Piatetski-Shapiro and Shahidi  in combination with the Langlands-Shahidi method to prove the conjecture for a globally generic automorphic representation $\pi$ when $\gb G$ is a quasi-split symplectic, unitary or special orthogonal group. For the latter, Arthur and his continuators used trace formulas to get more complete results,  not restricted to globally generic cuspidal automorphic representations of quasi-split groups in characteristic zero.

 Lomel\'i extended the converse theorem method to global function fields, establishing functoriality for globally generic automorphic representations of split classical groups and unitary groups. The present article complements this treatment of the quasi-split classical groups, over a function field $F$, by establishing the functoriality conjecture when $\gb G$ is a quasi-split  non-split even special orthogonal group, that we denote by ${\bf SO}_{2n}^*$, and $\pi$ a globally generic representation.

 \begin{thmen}Let $F$ be a global function field and $\pi=\bigotimes_x
'\pi_x$ a globally generic cuspidal automorphic representation of $\gb{SO}_{2n}^*(\mathbb A_F)$. Then, $\pi$ lifts to an irreducible automorphic representation $\Pi$ of $\gb{GL}_{2n}(\mathbb A_F)$. Furthermore, $\Pi$ can be expressed as an isobaric sum
\[\Pi=\Pi_1 \boxplus \cdots  \boxplus\Pi_d,\]
where each $\Pi_i$ is a unitary self-dual cuspidal automorphic representation of $\gb{GL}_{N_i}(\mathbb A_F)$ for some $N_i$, and where $\Pi_i\not\cong \Pi_j$ for $i\neq j$. Moreover, if we write $\Pi=\bigotimes_x'\Pi_x$, then for every place $x$ of $F$ and every  irreducible  generic  representation $\tau_x$ of $\gb{GL}_m(F_x)$ we have that
\begin{align*}
    \gamma(s,\pi_{x} \times \tau_{x},\psi_x) &=\gamma(s,\Pi_{x} \times \tau_{x},\psi_x),
\end{align*} 
where the $\gamma$-factors on the right are obtained by the Rankin-Selberg method and those on the left by the Langlands-Shahidi method, as extended by Lomel\'i to positive characteristic. 
\end{thmen}

As in Cogdell, Kim, Piatetski-Shapiro and Shahidi, and Lomel\'i, the method of proof uses the converse theorem and $L$-functions to construct an automorphic representation of $\gb{GL}_n(\mathbb A_F)$: we provide a proof of a twisted version in positive characteristic of the converse theorem of Cogdell and
Piatetski-Shapiro. To apply the converse theorem, one needs analytic properties of the Langlands-Shahidi $L$-functions, and to establish them we adapt Lomelí's arguments to our new case. We first obtain a  lift that has the desired properties at almost all places. Then further properties of partial $L$-functions give that there is a lift which is an isobaric sum of unitary cuspidal automorphic representations. We prove the compatibility between the local $\gamma$-factors of $\pi$ and the lift $\Pi$ at all places. Our lift is close to what is known as the strong lift, which states the compatibility of $L$-functions and $\varepsilon$-factors.

As an application of the functoriality and the validity of the Ramanujan conjecture for general linear groups established by L. Lafforgue, we prove the unramified Ramanujan conjecture for globally generic cuspidal automorphic representations of our classical group in positive characteristic. A strong lift would imply the Ramanujan conjecture at all places.
\begin{thmen}
Let $\pi=\bigotimes_x'\pi_x$ be a globally generic cuspidal automorphic representation of $\gb{SO}_{2n}^*(\mathbb A_F)$. Then, if $\pi_x$ is unramified, its Satake parameters have absolute value 1.
\end{thmen}

{\bfseries Acknowledgements}.  I'm grateful to Guy Henniart and Luis Lomel\'i for guiding and supporting me  through out the whole creation process of this article. I would also like to thank to both institutions that hosted my Ph.D. studies, in which this article was developed: Université Paris-Saclay and Pontificia Universidad Cat\'olica de Valpara\'iso. Finally, I want to thank the National Agency for Research and Development (ANID, Chile) for funding the Ph.D. project via the National Doctoral Scholarship 2017 21171037.

\section{Notation} We let $F$ be a function field in one
variable over a finite field or a  locally compact field of positive characteristic. Whenever $F$ is a function field in one
variable over a finite field, we denote by $|F|$ the set of places of $F$, by $F_x$ the completion at a place $x\in |F|$ and by $q_F$ or simply by $q$ the cardinality of its field of constants. We also denote the ring of ad\`eles of $F$ by $\mathbb A_F$. When $F$ is a locally compact field of positive characteristic, we denote by $q_F$ the cardinality of its residue field and by $|\cdot|_F$ or simply by $|\cdot|$ the absolute value of $F$. 

For an arbitrary algebraic group ${\bf H}$ over $F$, we denote ${\bf H}(R)$ its set of $R$-points, where $R$ is a $F$-algebra (usually, $R$ will be $F$ or $\mathbb A_F$). In order to reduce the size of some indices, we sometimes denote ${\bf H}(F)$ simply by $H$.

Let $\gb G$ be a quasi-split reductive group over $F$ with a maximal split subtorus $\gb S$. Fix a Borel subgroup ${\bf B}$ containing ${\gb S}$.  We denote the set of (relative) simple roots of ${\bf S}$ in ${\bf B}$ by $\Delta$. The Weyl group $N_{\gb G}(\gb S)/Z_{\bf G}({\bf S})(F)$ is denoted by $W^{\bf G}$.

If ${\bf P}$ is a parabolic subgroup of ${\bf G}$ containing ${\bf S}$, then the relation ${\bf P}={\bf M}{ \bf N}$ will also mean that ${\bf M}$ and ${\bf N}$ are the Levi subgroup containing ${\bf S}$ and ${\bf N}$ is the unipotent radical of ${\bf P}$ respectively. Furthermore, we recall that there is an inclusion preserving bijection between subsets of $\Delta$ and parabolic subgroups containing ${\bf B}$. We denote the parabolic subgroups associated to $\theta\subset \Delta$ via that bijection by ${\bf P}_\theta$.

Suppose that $F$ is a locally compact field of positive characteristic. Let ${\bf P = MN}$ be a parabolic subgroup of ${\bf G}$ containing ${\bf B}$, then we denote by $i^G_P$ or $i_{{\bf P}(F)}^{{\bf G}(F)}$, the normalized parabolic induction functor from ${\bf P}(F)$ to ${\bf G}(F)$. More precisely, if $(\sigma,W)$ is a smooth representation of $M$, then we denote by $i_P^G(W)$ the space of all locally constant functions $f:G\to W$ such that  $f(mng)=\delta^{-1/2}_P(m)\sigma(m)f(g)$ for $m\in M,n\in N,g\in G$, where $\delta_P$ is the modulus character of $P$. The action of $G$ on $i_P^G(W)$ is given by $g\cdot f(x)=f(xg)$.
\section{Quasi-split even special orthogonal group}We start with the general construction of $\gb{SO}(q)$, for an even dimensional non-degenerate quadratic space $Q=(V,q)$. Two references for this section are \cite[Appendix C]{Bgrp} and \cite{Kn}.
\subsection{Definitions}\label{subsec:SO*}
 First we recall some definitions: for an $F$-algebra $R$, an $R$-quadratic space is a pair $(V,q)$ consisting of a finite free $R$-module $V$  and a quadratic form $q:V\to R$  i.e.
\begin{enumerate}
    \item $q(rv)=r^2q(v)$, for all $r\in R$ and $v\in V$,
    \item the map $B_q\colon V\times V \to  R$, defined by $B_q(x,y)=q(x+y)-q(x)-q(y)$, is $R$-bilinear. 
\end{enumerate} 
The orthogonal group $\gb O(q)$ for a general $F$-quadratic space $(V,q)$ over $F$, is a closed subscheme of $\gb{GL}(V)$, which represents the functor
\[R\mapsto \{g\in \gb{GL}(V\otimes_FR)\colon q_{R}(gx)=q_{R}(x) \text{ for all } x\in V\otimes_FR\}.\]
The special orthogonal group $\gb{SO}(q)$ is defined as the kernel of the Dickson morphism $D_q$. The Dickson morphism $D_q$ is a smooth surjection, identifying $\mathbb{Z}/2\mathbb{Z}$ with $\gb O(q)/\gb{SO}(q)$. The special orthogonal group $\gb{SO}(q)$ is connected, smooth and reductive of dimension $n(n-1)/2$. Its center is the group of $2$-root of unity $\mu_2$ as a group scheme.

We  consider two families of quadratic spaces: for $n\geq 1$, let $Q_n=(F^{2n},q_n)$ be the quadraitc space with
\[q_n(x_1,\dots,x_{2n})= x_1x_{2n}+\cdots+ x_{n}x_{n+1}. \]
And for a separable quadratic extension $E$ of $F$ and $n\geq 1$ let $Q_{E,n}=(F^{n-1}\oplus E \oplus F^{n-1}, q_{E,n})$, where \[q_{E,n}(x_1,\dots,x_{n-1},x,x_{n+2},\dots,x_{2n})=x_1x_{2n}+\cdots+ x_{n-1}x_{n+2}+\operatorname{N}_{E/F}(x)\]
 and $\operatorname{N}_{E/F}$ is the norm map.

To simplify the notation, when the base field $F$ is clear from  context, we let \[\gb{SO}_{2n}\coloneqq\gb{SO}(q_{n}).\]
Also, when the separable quadratic extension $E$ of $F$ is clear from  context, we let \[\gb{SO}_{2n}^*\coloneqq \gb{SO}(q_{E,n}).\]

\begin{rmk}\label{rmk:split&odd}
\begin{itemize}\item[]
    \item If we allow $E$ to be $F\times F$, we get that $\gb{SO}(q_{F\times F,n})\cong \gb{SO}_{2n}$. In this case we have that $N_{F\times F/F}(x,y)=(xy,xy)$.
    \item There is a similar construction of special orthogonal groups for odd dimensional non-degenerate quadratic spaces. In particular, we get the (split) odd special orthogonal groups ${\bf SO}_{2n+1}$.
\end{itemize}
\end{rmk}
\subsection{Structure} \label{section:structure} We now recall some structural properties of the group ${\bf SO}_{2n}^*={\bf SO}_{2n}(q_{E,n})$. First, since $E\otimes_F E\cong E\times E$ and Remark \ref{rmk:split&odd}, we have that the group $\gb{SO}(q_{E,n})$ is an $F$-form of the split group ${\bf SO}_{2n}$, which splits over $E$. Thus, the absolute root system is of type $D_{n}$. 

   We note that the orthogonal decomposition $Q_{E,n}=H_1\perp \cdots \perp H_{n-1}\perp (E,\operatorname{N}_{E/F})$, for $n-1$ hyperbolic planes $H_i=(F e_i\oplus Fe_{2n-i+1}, x_ix_{2n-i+1})$ and an (anisotropic) non-degenerate quadratic space $(E,\operatorname N_{E/F})$, allows us to obtain the following subtori of $\gb{SO}(q_{E,n})$
\[\gb S\cong \prod_{i=1}^{n-1} \gb{SO}(H_i) \quad \& \quad \gb T\cong \gb S\times\gb{SO}(E,\operatorname N_{E/F}).\]
Since the dimension of $\gb S$ is $n-1$ and the dimension of $\gb T$ is $n$, we get that $\gb S$ is a maximal split $F$-torus of $\gb{SO}(q_{E,n})$  and $\gb T$ is a maximal $F$-torus of $\gb{SO}(q_{E,n})$. 

The $F$-points of $\gb{SO}_2$ and $\gb{SO}(E,N_{E/F})$ can be identified with the multiplicative group $F^\times$ and the group $E^1$ of norm one elements of $E^\times$,  respectively. 
More precisely, we can describe the action of $t=(x_1,\dots,x_{n-1},x)\in (F^\times )^{n-1}\times  E^1\cong \gb T(F)$ on  the $F$-vector space $Q_{E,n}$ by $t\cdot e_i=x_ie_i$ for $1\leq i\leq n-1$, $t\cdot l= xl$ for $l \in E$ and $t\cdot e_{2n-i+1}=x_i^{-1}e_{2n+1-i}$ for $1\leq i \leq n-1$.
Moreover, \[Z_{\gb G}(\gb S)=\gb T.\]

We also note that,  by restricting $q_{E,n}$ to $H_1\perp \cdots \perp H_{n-1}\perp F\cdot 1$, we get the split quadratic form of $2(n-1)+1$ variables
\[x_1x_{2n}+\cdots+ x_{n-1}x_{n+2}+x^2.\]
Thus, we get a copy of $\gb{SO}_{2n-1}\cong\gb{SO}(W\perp F\cdot 1)$ inside $\gb{SO}(q_{E,n})$, where $W=H_1\perp \cdots \perp H_{n-1}$. And, the (relative) root system $\Phi(\gb{SO}(q_{E,n}),\gb S)=\Phi(\gb{SO}(W\perp F\cdot 1),\gb S)$ is of type $B_{n-1}$.

Let $\gb B$ be the Borel subgroup of $\gb{SO}(q_{E,n})$ containing $\gb S$ and stabilizing the standard flag $(W_{\text{std}},F_{\text{std}})$, where $W_{\text{std}}=\operatorname{span}(e_0,\dots,e_{n-1})$ and $F_{\text{std}}=\{Fe_1,Fe_2 \oplus Fe_1,\dots, W_{\text{std}} \}$.  We have that
\[\gb B=\gb T\ltimes R_{u,F}(\gb B),\]
  where $R_{u,F}(\gb B)$ is the radical unipotent subgroup of ${\bf B}$, and that $\gb{SO}(q_{E,n})$ is quasi-split.

\subsection{Relative rank one}
 We relate $\operatorname{Res}_{E/F}\gb{SL}_2$ to the simply connected cover of $\gb{SO}(q_{E,2})$. Let $\operatorname{Gal}(E/F)=\{1,\sigma\}$. Then, we note that $Q_{E,2}$ is isomorphic to the quadratic  space $(H,q)$ of Hermitian $2\times 2$ matrices, i.e. $2\times2$ matrices $(x_{i,j})$ with entries in $E$ such that $x_{i,j}=\sigma(x_{j,i})$ for  every $1\leq i,j\leq 2$, with quadratic form $q=-\det$, via 
 \[(x_1,x_2,x)\mapsto \begin{pmatrix}
 -x_1 & x \\
 \sigma(x) & x_2
 \end{pmatrix}.\] Furthermore, we have an action of $\operatorname{Res}_{E/F}\gb{SL}_2$ on $H$, via 
 \[a \mapsto g a g^*,\]
 where $g^*={}^t\sigma(g)$. We observe that $q(a)=q(gag^*)$ and $\det(a \mapsto g a g^*)=1$. Thus the action gives us a morphism 
 \[\operatorname{Res}_{E/F}\gb{SL}_2 \to \gb{SO}'(q)=\ker (\det|_{\gb O(q)}).\]
 As $\operatorname{Res}_{E/F}\gb{SL}_2$ is connected and $\gb O(q)/\gb{SO}(q)=\mathbb Z/2\mathbb Z $, this morphism factors through $\gb{SO}(q)\cong \gb {SO}(q_{E,2}) $. Finally as the kernel of this morphism is $\mu_2$ and the dimensions of the groups $\gb {SO}(q_{E,2})$ and $\operatorname{Res}_{E/F}\gb{SL}_2$ are the same, we have
 \[1\to \mu_2\to \operatorname{Res}_{E/F}\gb{SL}_2\to \gb{SO}(q_{E,2})\to 1.  \label{rnk1}\]

\subsection{$L$-group} We finish this section by describing the $L$-groups of ${\bf SO}^*_{2n}$, ${\bf SO}_{2n}$ and ${\bf GL}_m$. For that, let us fix a separable closure $F_s$ of $F$ and a quadratic separable extension $E$ of $F$ contained in $F_s$. Denote  $\Gamma_F=\operatorname{Gal}(F_s/F)$. Finally, we denote by ${\bf SO}(q)_{F_s}$ and ${\bf T}_{F_s}$ the extension of scalar to $F_s$ of ${\bf SO}(q)$ and ${\bf T}$, respectively.

Let us consider the split special orthogonal group $\gb{SO}_{2n}$ over $\mathbb C$. We choose a split maximal torus $\gb T_n$ and a Borel subgroup $\gb B_n$ in the following way
\[\gb T_n(\mathbb C)=\{t=\operatorname{diag}(t_1,\dots,t_{n},t_{n}^{-1},\dots,t_1^{-1})\colon t_i \in \mathbb C^\times, 1\leq i\leq n\}\]
and
\[\gb B_n=\gb T_n\rtimes\{M(u) h(L)\colon v^tLJ^tv=0 \text{ for all }v\in \mathbb C^n\},\]
the subgroup of upper triangular matrices in $\gb{SO}_{2n}$, where $J = (\delta_{i,n+1-i})$ with $\delta$ the Kronecker's delta is the  $n\times n$ matrix with $1$'s along the anti-diagonal, and for upper triangular unipotent $u\in\gb{GL}_n$ and $L$ an $n\times n$ matrix,
\[M(u)=\begin{pmatrix}u & 0_n \\
0_n & J({}^tu)^{-1}J
\end{pmatrix} \quad \& \quad h(L)\coloneqq \begin{pmatrix}1_n & L \\
0_n & 1_n
\end{pmatrix} .\]
We observe that the root datum associated to $(\gb{SO}_{2n},\gb T_n)$ is isomorphic to the dual root datum $R^\vee$ of  $(\gb{SO}(q_{E,n})_{F_s},\gb T_{F_s})$. We choose a pinning of $(\gb{SO}_{2n},\gb T_n,\gb B_n)$ corresponding  to the based root data  $(R^\vee,\Delta^\vee)$ in the following manner, \[X_{\alpha_i^\vee}=\begin{pmatrix}E_{i,i+1} & 0_n \\
0_n & -E_{i,i+1}
\end{pmatrix}
\in \operatorname{Lie}(\gb B_n) \text{ for } 1\leq i \leq n-1,\]  
where $E_{i_0,j_0}=(\delta_{i_0,j_0})$, 
 and
\[X_{\alpha_n^\vee}=h\begin{pmatrix}  &&0_{n-2}\\
1 & \phantom{-}0& \\
0&-1& &\end{pmatrix}\in \operatorname{Lie}(\gb B_n).\]
Finally, using the equivalence of categories given in \cite[Theorem 6.1.17]{Bgrp}, we  identify $\gb{SO}_{2n}$ over $\mathbb C$ with the dual group of ${\bf SO}_{2n}$ over $F$. We thus  fix this identification.

(\textit{The $*$-action}). Since $\gb{SO}_{2n}^*=\gb{SO}(q_{E,n})$ is non-split, we have a non-trivial action of $\Gamma_F$ on $(R^\vee,\Delta^\vee)$. Moreover, if we denote  
\[w=\begin{pmatrix}1_{n-1} & & & \\
 &0 & 1 & \\
 & 1&0 & \\
 &  & &1_{n-1}
\end{pmatrix},\]
we have that $(g\mapsto wgw^{-1})\in\operatorname{Aut}(\gb{SO}_{2n},\gb T_n,\{X_a\}_{a\in \Delta^\vee})$ corresponds, via the equivalence of categories \cite[Theorem 6.1.17]{Bgrp}, to the non-trivial automorphism in $\operatorname{Aut}(R^\vee,\Delta^\vee)$. Thus, the induced action of $\Gamma_F$ is given by 
\begin{align*}
    \varphi \colon \Gamma_F&\to \operatorname{Aut}(\gb{SO}_{2n}) \\
    \tau &\mapsto \begin{cases}(g\mapsto wgw^{-1}) &\text{ if } \tau \not \in \operatorname{Gal}(F_s/E)\\
  (g \mapsto g) &\text{ if } \tau \in \operatorname{Gal}(F_s/E)\end{cases}.
\end{align*}
Putting these identifications together, we have that the $L$-group of ${\bf SO}_{2n}^*$ over $F$ can be identified with the following semidirect product
\[\gb{SO}_{2n}(\mathbb C)\rtimes_\varphi \Gamma_F.\]

In the case of the split groups ${\bf GL}_m$ and ${\bf SO}_{2n}$ over $F$, their $L$-groups can be identified with the following direct product
\[{\bf GL}_m(\mathbb C)\times \Gamma_F \quad(\text{resp. } {\bf SO}_{2n}(\mathbb C)\times \Gamma_F).\]
\section{Langlands functoriality and Converse theorem}
 In this section we state the main result of this article and we start with the first steps of the proof of generic functoriality. The strategy is inspired from \cite{Cogqsplit} and \cite{L15}.

Suppose that $F$ is a function field in one
variable over a finite field. We let for $x\in |F|$, $\gb{G}_x= \gb{G}_{F_x}$ the  group obtained from $\gb G$ by extending scalars along the inclusion $F\hookrightarrow F_x$. We choose a separable closure $F_{x,s}$ of $F_x$ and an embedding $F_s\hookrightarrow F_{x,s}$  for every $x\in |F|$ that extends $F\hookrightarrow F_x$. We write $\Gamma_F=\operatorname{Gal}(F_s/F)$, $\Gamma_{F_x}=\operatorname{Gal}(F_{x,s}/F_x)$ for every $x\in |F|$ and we let $I_{F_x}$ be the inertia subgroup of $\Gamma_{F_x}$, for every $x\in |F|$. These choices give us an injection $\Gamma_{F_x} \hookrightarrow \Gamma_F$.

Now the restriction along $\Gamma_{F_x}\hookrightarrow \Gamma_F$ induces a (continuous) group homomorphism 
from ${}^LG_x$ to ${}^LG$, that fits  into a commutative diagram
\[\begin{tikzcd}{}^LG_x \arrow[d]\arrow[r] & {}^LG. \arrow[d] \\
\Gamma_{F_x} \arrow[r]& \Gamma_F 
\end{tikzcd}\]
 Now given an $L$-homomorphism $\rho\colon {}^{L}G \to {}^{L}H$, for every place $x$ we can form an $L$-homomorphism   $\rho_x\colon {}^LG_x \to {}^LH_x$. 

\subsection{Satake parametrization} For every $x\in |F|$, we denote by $W_{F_x}$ the Weil group of $F_x$. We fix a geometric Frobenius element $\mathrm{Fr}_x\in W_{F_x}$. We denote by $\Phi(G_x)$ the set of group morphisms \cite[Section 8]{CorB} \[\phi\colon W_{F_x}'=W_{F_x}\times \operatorname{SL}_2(\mathbb C)\to {}^LG_x,\]
such that $\phi(\operatorname{Fr}_x)$ is semi simple, $\phi|_{I_{F_x}}$ is continuous, $\phi|_{\operatorname{SL}_2(\mathbb C)}$ is algebraic and $\phi$ is relevant, i.e. if the image of $\phi$ is contained in a Levi subgroup of ${}^LG_x$ then it is the $L$-group of a Levi subgroup of $\gb G_x$, modulo ${}^L
G_x^\circ$-conjugacy classes of parameters.
Moreover, when $\phi|_{I_{F_x}}$  and $\phi|_{\operatorname{SL}_2(\mathbb C)}$ are trivial, $\phi$ will be called unramified. We denote by $\Phi_{unr}(G_x)$ the set of these classes.

Thus, we also obtain  a bijection between unramified representations and unramified parameters \cite[Section 9.5]{CorB},
\begin{align*}
    \Pi_{unr}(G_x) &\to \Phi_{unr}(G_x).\\
    \pi &\mapsto  \phi_\pi
\end{align*}

\subsection{Functoriality conjecture}
Let \[\rho\colon {}^LG\to {}^LH\] be an $L$-homomorphism  between the $L$-groups of a reductive group ${\bf G}$ and a quasi-split reductive group ${\bf H}$ over $F$, respectively. Let $\pi=\bigotimes_x'\pi_x$ be a cuspidal automorphic representation  of $\gb G(\mathbb A_F)$. A (weak) lift or a transfer of $\pi$ through $\rho$ is an automorphic representation $\Pi=\bigotimes_x' \Pi_x$  of $\gb H(\mathbb{A}_F)$, such that  the following commutative diagram commutes
\[\begin{tikzcd}^{L}G_x \arrow[rr, "\rho_x"]& & ^{L}H_x \\
&W'_{F_x} \arrow[ul,"\phi_{\tau_x}"] \arrow[ur,"\phi_{\pi_x}"'] &
\end{tikzcd}\]
for all places $x$ such that $\gb H_x,\gb G_x,\pi_x$ and $\Pi_x$ are unramified. When  $\rho$ is clear from the context we will just say that it is a lift or a transfer of $\pi$. The functoriality conjecture states that such $\Pi$ always exists, for every  cuspidal automorphic representation $\pi$ and $L$-homomorphism $\rho$.

Our main objective  is to prove the existence of such a transfer  for ${\bf G}={\bf SO}_{2n}^*$, ${\bf H}={\bf GL}_{2n}$, $\pi$ globally generic and $\rho$ given by 
\begin{align}
  \rho_{2n}^*: \gb{SO}_{2n}(\mathbb{C})\rtimes \Gamma_F &\to \gb{GL}_{2n}(\mathbb{C})\times \Gamma_F  \nonumber \\
  (g,\tau)&\mapsto \begin{cases}(gw,\tau) &\text{ if } \tau \not \in \operatorname{Gal}(F_s/E)\\
  (g,\tau) &\text{ if } \tau \in \operatorname{Gal}(F_s/E)\end{cases}\label{L-hom}. 
\end{align}

\subsection{Candidate lift}\label{subsec:candidate}
For every place $x$ we choose a character $\lambda_x$ of $\gb T(F_x)$ such that for $\pi_x$ unramified it is the character obtained by the Satake parametrization  and for $\pi_x$ ramified, $i_{B}^{\gb{SO}_{2n}^*(F_x)}(\lambda_x)$ has an (irreducible) generic subquotient $\pi'_{\lambda_x}=\pi_x'$ with the same central character as $\pi_x$ (See for example \cite[Section 4.2]{Cogqsplit}). Applying the local Langlands correspondence for tori, from $\lambda_x$ we get  $\phi_{\lambda_x}:W_{F_x}'\to \gb T_n(\mathbb C)\rtimes \Gamma_{F_x}$.

Let $i_x:\gb T_n(\mathbb C)\rtimes \Gamma_{F_x}\hookrightarrow \gb{SO}_{2n}(\mathbb{C})\rtimes \Gamma_{F_x}$ be the inclusion homomorphism. Then, applying the local Langlands correspondence for general linear groups to \[\rho^*_{2n,x}\circ i_x \circ \phi_{\lambda_x}\colon W_{F_x}'\to \gb{GL}_{2n}(\mathbb C)\times \Gamma_{F_x},\]we find an admissible representation $\Pi_x$ of $\gb{GL}_{2n}(F_x)$.
We put $\Pi=\bigotimes_x'\Pi_x$, which is an irreducible admissible representation of $\gb{GL}_{2n}(\mathbb{A}_F)$. We call $\Pi$  (resp. $\Pi_x$) a candidate lift or candidate transfer of $\pi$ (resp. $\pi_x$). The representation $\Pi$ is not necessarily a transfer because it is not necessarily an automorphic representation of ${\bf GL}_{2n}(\mathbb A_F)$. But, together with the converse theorem of Section \ref{subsec:converse}, we will use it  in order to construct in Section \ref{section:genericfunc} the desired lift.

\subsection{A description of $\Pi_x$}\label{subsec:description}Now for every place $x$, we will give a description of $\Pi_x$.   First, let us note that by definition \[\gb{SO}_{2n}^*(F_x) =\gb{SO}(q_{E_x,n})(F_x)
,\] where $E_x\coloneqq E\otimes_F F_x$ is a degree two \'etale algebra over $F_x$. Thus, it is either a product of two (separable and trivial) fields  extensions or a (separable) field extension over $F_x$.  Let us concentrate on the case when $E_x$ is a quadratic (separable) extension of $F_x$ (i.e. $x$ is an inert place), for which we have an embedding $E_x\hookrightarrow  F_{x,s}$, coming from the one fixed in the beginning of this chapter $F_s\hookrightarrow F_{x,s}$. 

Now let us consider torus  $\gb T'=\mathbb G_m^{2n-2}\times \operatorname{Res}_{E_x/F_x}(\mathbb G_m)$, where $\mathbb G_m$ is the multiplicative group. We have an isomorphism \[E_x\otimes_{F_x}F_{x,s}\cong \prod_{\sigma\in \operatorname{Hom}(E_x,F_{x,s})} F_{x,s}.\]
This leads us to an isomorphism  \[X_*(\gb T'_{F_{s,x}})=X_*(\mathbb G_m^{2n-2})\times X_*(\operatorname{Res}_{E_x/F_x}(\mathbb G_m)_{F_{s,x}})\cong \mathbb Z^{2n-2}\times \mathbb Z^2,\] where $X_*$ denotes the groups of cocharacters and the non-trivial action of the second factor, $\mathbb Z^2$,  is given by 
\begin{align*}
    \Gamma_{F_x}&\to \operatorname{Aut}(\mathbb Z^2) \\
    \tau &\mapsto \begin{cases}(a_1,a_2)\mapsto (a_2,a_1) &\text{ if } \tau \not \in \operatorname{Gal}(F_{x,s}/E_x)\\
  (a_1,a_2) \mapsto (a_1,a_2) &\text{ if } \tau \in \operatorname{Gal}(F_{x,s}/E_x)\end{cases}.
\end{align*}
If we denote by $\gb D_{2n}\subset \gb{GL}_{2n}$ the  maximal diagonal torus, we can identify 
\[\gb D_{2n}(\mathbb C)\rtimes \Gamma_{F_x}\cong {}^LT'_x,\]
where the action on the left hand side is given by conjugation by $w$.

Now thanks to this description we can construct the following embeddings  \begin{align*}
 \iota_x \colon \gb T_n(\mathbb C)\rtimes \Gamma_{F_x} &\hookrightarrow \gb D_{2n}(\mathbb C)\rtimes \Gamma_{F_x}\cong{}^LT'_x \\
 (t,\tau) &\mapsto ((t,t^{-1}),\tau)
\end{align*}
and
\begin{align*}
   \gb D_{2n}(\mathbb C)\rtimes \Gamma_{F_x}&\hookrightarrow  \gb{GL}_{2n}(\mathbb C)\times\Gamma_{F_x}\\
    (d,\tau)&\mapsto \begin{cases}(dw,\tau) &\text{ if } \tau \not \in \operatorname{Gal}(F_{x,s}/E_x)\\
  (d,\tau) &\text{ if } \tau \in \operatorname{Gal}(F_{x,s}/E_x)\end{cases}.
\end{align*}
We can thus factor $\rho_{2n,x}^*\circ i_x$, 
 \[\begin{tikzcd}
  {}^LT_x\ar[rr, hook, "\rho_{2n,x}^*\circ i_x"]\ar[rd, hook]& & \gb{GL}_{2n}(\mathbb C)\times\Gamma_{F_x}. \\
&  {}^LT'_x \ar[ur, hook]&
 \end{tikzcd}\] 
 The definition of the candidate lift and previous factorization lead us to look at
 \[\begin{tikzcd}\Phi(\gb T_x)\ar[r]\ar[d]\arrow[dr, phantom, "\raisebox{.5pt}{\textcircled{\raisebox{-.9pt} {1}}}"]& \Phi(\gb T'_x)\ar[r] \ar[d] \arrow[dr, phantom, "\raisebox{.5pt}{\textcircled{\raisebox{-.9pt} {2}}}"]& \Phi(\gb{GL}_{2n,x}) \ar[d] \\
\Pi(\gb T_x) \ar[r] & \Pi( \gb T'_x)\ar[r] & \Pi(\gb{GL}_{2n,x}),
\end{tikzcd}
\]
where the first two vertical arrows are the ones given by the local Langlands correspondence for algebraic tori,  the third  one is  given  by the local Langlands correspondence for general linear group and the upper horizontal arrows are the ones obtained from composition with $\rho_{2n,x}^*\circ i_x$ (and its factorization).

Let $\lambda_x=\chi_{1,x}\otimes\dots\otimes\chi_{n-1,x}\otimes\chi_{n,x}$ be a character of $\gb T(F_x)=(F^\times_x)^{n-1}\times \gb{SO}_2^*(F_x)$, where $\chi_{i,x}$ is a character of $F_x^\times$  for $1\leq i \leq n-1$ and $\chi_{n,x}$ is a character of $\gb{SO}^*_2(F_x)$. The image of  \textcircled{\raisebox{-.9pt} {1}} is  \[\Lambda_x=\chi_{1,x}\otimes\dots\otimes\chi_{n-1,x}\otimes\mu_{n,x}\otimes\chi_{n-1,x}^{-1}\otimes\dots\otimes\chi_{1,x}^{-1},\]
 where $\mu_{n,x}:E_x^\times \to \mathbb{C}^\times$ is the character obtained from $\chi_{n,x}\colon\gb{SO}^*_2(F_x)\to\mathbb{C}^\times$ via \[\Phi(\gb{SO}_2^*)\to \Phi(\operatorname{Res}_{E_x/F_x}\mathbb G_m).\]
 To specify $\mu_{n,x}$, let $\phi_x$ be the parameter of $\chi_{n,x}=\pi_{\phi_x}$. Using \cite[p. 235]{RLab}, we have that
 \[\pi_{(\phi_{x,E_x})}
\colon \gb{SO}_2^*(E_x) \xrightarrow{\operatorname{Norm}}\gb{SO}_2^*( F_x)\xrightarrow{\pi_{\phi_x}}\mathbb C^\times,\]
 where $\pi_{(\phi_{x,E_x})}$ is the representation of $\gb{SO}_2^*(E_x)$ corresponding to the parameter in $\Phi(\mathbb G_{m,E_x})$ obtained from $\phi_x$ via the restriction $ E_x\to W_{E_x/F_x}$. We can give a description of this construction using the identification between  $\gb{SO}_2^*$ and the norm one elements of $\operatorname{Res}_{E_x/F_x}\mathbb G_m$ in Section \ref{section:structure}. Indeed, we have that $\gb{SO}^*_2(F_x)=E_x^1$, $\gb{SO}^*_2(E_x)=E_x^\times$ and the corresponding norm map, also denoted by $\operatorname{Norm}$, is given by
 \begin{align*}
     \operatorname{Norm}\colon E_x^\times &\to E_x^1 \\
     x &\mapsto x\sigma(x)^{-1},
 \end{align*}
 and thus
 \begin{align*}
   \pi_{(\phi_{x,E_x})}
\colon E_x^\times &\to \mathbb C^\times\\
     x&\mapsto \chi_{n,x}(x\sigma(x)^{-1}).
 \end{align*}
On the other hand, we recall that we have an isomorphism (Shapiro's Lemma) \cite[Proposition 8.4]{CorB}
 \begin{align*}
 \Phi(\operatorname{Res}_{E_x/F_x}\mathbb G_m)&\to \Phi(\mathbb G_{m,E_x})\\
(W_F\to I_{\Gamma_{E_x}}^{\Gamma_{F_x}}\widehat{\mathbb G_m}(\mathbb C)\rtimes \Gamma_{F_x})&\mapsto (W_{E_x}\to \widehat{\mathbb G_m}(\mathbb C) \times \Gamma_{E_x}),
\end{align*}
where $I_{\Gamma_{E_x}}^{\Gamma_{F_x}}$ is the functor of induction from $\Gamma_{F_x}$ to $\Gamma_{E_x}$. The image of  $\iota_x\circ\phi_x\in \Phi(\operatorname{Res}_{E_x/F_x}\mathbb G_m) $ through this isomorphism is the parameter in $\Phi(\mathbb G_{m,E_x})$ obtained from $\phi_x$ via the restriction $ E_x\to W_{E_x/F_x}$. Therefore, since $\mu_{n,x}$ is the character of $E_x^\times$ corresponding to the parameter $\iota_x\circ\phi_x$ , we have that \begin{equation}
\mu_{n,x}=\pi_{(\phi_{x,E_x})}.     
\end{equation}

Now for the second square, first we look at the image of the parameter corresponding to $\mu_{n,x}$ via
 \[\Phi(\operatorname{Res}_{E_x/F_x}\mathbb G_m)\to \Phi(\gb{GL}_2).\]
 First, using again the identification given above (Shapiro's Lemma),
  \[\Phi(\operatorname{Res}_{E_x/F_x}\mathbb G_m) \to \Phi(\mathbb G_{m,E_x})\to \Pi(E^\times),\]
 we have that 
 the image of $\mu_{n,x}$ corresponds to the Weil-Deligne representation \[(\operatorname{Ind}_{E_x/F_x}(\mu_{n,x}\circ \operatorname{ar}_{E_x}), 0),\]
 where $\operatorname{ar}_{E_x}=r_{E_x}^{-1}$ is the reciprocity map \cite[IV. (6.3)]{Neu} and $\operatorname{Ind}_{E_x/F_x}$ is the functor of smooth induction from $\Gamma_{F_x}$ to $\Gamma_{E_x}$. Now using the local Langlands correspondence for $\gb{GL}_2$ (\cite[Chapter 8]{BHGL2}), we get our admissible irreducible representation of $\operatorname{GL}_2={\bf GL}_2(F_x)$: 
 \[\Pi_{\mu_{n,x}}=\begin{cases} i_{B_2}^{GL_2}(\nu_{n,x}\otimes\varkappa_x\nu_{n,x}), & \text{if  $\mu_{n,x}=\nu_{n,x}\circ \operatorname{N}_{E_x/F_x}$, for some characater $\nu_{n,x}$ of $F_x^\times$, }
 \\
 \pi_{\mu_{n,x}} & \text{otherwise,}
 \end{cases} \label{liftex}\]
 where $\varkappa_x=\det (\operatorname{Ind}_{E_x/F_x}1_{E_x})$ and $\pi_{\mu_{n,x}}$ is constructed in \cite[Chapter 8]{BHGL2}, from a the character $\mu_{n,x}\colon E^\times\to \mathbb C^\times$. 
 
 Putting all this together we get an expression for $\Pi_x$. In particular for $\lambda_x=(\chi_{1,x},\dots,\chi_{n-1,x},1)$ unramified  we have  that $\Pi_x$ is  the constituent of 
 \[i^G_B(\chi_{1,x}\otimes \dots\otimes\chi_{n-1,x}\otimes 1 \otimes \varkappa_x\otimes \chi_{n-1,x}^{-1}\otimes \dots\otimes \chi_{1,x}^{-1}), \label{liftu}\]  that has a nonzero vector fixed under the special maximal compact subgroup $\gb{GL}_{2n}(\mathcal O_x)$.  
 
 Finally, we note that this construction gives us that the central character of $\Pi_x$ is \[ \varkappa_x\mu_{n,x}|_{F_x^\times}=\varkappa_x.\] Thus, the global character \begin{equation}
 \varkappa=\otimes_x ' \varkappa_x, \label{central}    
 \end{equation} is trivial on $F^\times$.

\subsection{Rankin-Selberg factors}\label{RS} To obtain an  automorphic representation,  we use as starting point the converse theorem. This theorem allows us to translate  the existence of an automorphic representation $\Pi$ of $\gb{GL}_{2n}(\mathbb{A}_F)$ to the existence of an admissible representation $\gb{GL}_{2n}(\mathbb{A}_F)$ with the right properties on its $L$-functions. We  present the $L$-functions and then we will describe the properties needed for this translation: the converse theorem.

Recall that we denote by $q_F$ or simply by $q$ the cardinality of its field of constants and by $q_{F_x}$ the cardinality of the residue field of $F_x$.

   (\textit{Local}). For every $x\in |F|$, let us consider a smooth non-trivial character   $\psi_x\colon F_x\to \mathbb C^\times$ and a pair of irreducible smooth representations  $\pi_x$ and $\pi_x'$ of  $\gb{GL}_n(F_x)$ and $\gb{GL}_{m}(F_x)$, respectively. We define as in \cite{RSc} the (local) Rankin-Selberg $L$-functions, $\varepsilon$-factors and $\gamma$-factors: 
\begin{align*}
    L(s,\pi_x\times\pi_x'), \quad \varepsilon(s,\pi_x\times\pi_x',\psi_x)\quad \& \quad \gamma(s,\pi_x\times\pi_x',\psi_x).
\end{align*}
\begin{rmk}\label{RSLS}
We remark that from \cite[Corollary 3.8]{HL}, in the where case where $\pi_x,\pi_x'$ are  generic representations of $\gb{GL}_n(F_x)$ and $\gb{GL}_m(F_x)$ respectively, the polynomials $L(s,\pi_x\times \pi_x')$ and $\varepsilon(s,\pi_x\times\pi_x',\psi_x)$ coincide with  the corresponding factors $L(s,\pi_x\otimes \Tilde \pi_x',r)$ and $\varepsilon(s,\pi_x\otimes\Tilde \pi_x',\psi_x,r)$ obtained using the Langlands-Shahidi method (Section \ref{subsec:LocalLS}). The maximal parabolic  subgroup considered contains the upper triangular matrices and its Levi subgroup is isomorphic to $\gb{GL}_n\times \gb{GL}_m$.
\end{rmk}
(\textit{Global}). Let $\psi=\bigotimes_{x}'\psi_x: \mathbb A_F/F \to \mathbb C$ be a continuous non-trivial character,  $\pi=\bigotimes_x' \pi_x$  an admissible representation of $\gb{GL}_r(\mathbb A_F)$ and  $\pi'=\bigotimes_x' \pi_x'$  an admissible representation of $\gb{GL}_{r'}(\mathbb A_F)$. We assume that $\pi_x$ and $\pi_x'$ are irreducible. 
We define the (global) Rankin-Selberg $L$-functions, $\varepsilon$-factors and $\gamma$-factors:
\begin{align*}
  L(s,\pi\times\pi')&=\prod_x L(s,\pi_x\times\pi'_x), \text{ as a formal power series in } q^{-s},\\
  \varepsilon(s,\pi\times\pi',\psi)&
  =\prod_x \varepsilon(s,\pi_x\times\pi_x',\psi_x) \text{ monomial function in } q^{-s}.
\end{align*}

\subsection{Converse Theorem}\label{subsec:converse}
In this Section, based on \cite{Cogcon} and \cite{L}, we provide a proof of the twisted version of the converse theorem found in \cite[Section 2]{Cogqsplit} for an admissible irreducible representation in positive characteristic. This result is stated in \cite[Theorem 8.1]{L09}, and we now take the opportunity to provide a proof.

Let $S$ be a finite subset of $|F|$. For each integer $m$, let 
\[\mathcal{A}_0(m)=\{\tau\mid \tau \text{ is a cuspidal automorphic representation of } \gb{GL}_{m}(\mathbb A_F)\},\]
and 
\[\mathcal{A}_0^S(m)=\{\tau\in \mathcal{A}_0(m)\mid \tau_v \text{ is unramified for all }v\in S\}.\]
For $n\geq 2$, we set 
\[\mathcal{T}(n-1)=\coprod_{m=1}^{n-1}\mathcal{A}_0(m) \quad \text{ and } \quad \mathcal{T}^S(n-1)=\coprod_{m=1}^{n-1}\mathcal{A}_0^S(m). \]
 If $\eta$ is a continuous character $F^\times\setminus \mathbb{A}^\times_F$, set
 \[\mathcal{T}(S;\eta)=\mathcal{T}^S(n-1)\cdot \eta=\{\tau=(\tau'\cdot \eta) \colon \tau'\in \mathcal{T}^S(n-1)\},\]
 where  $\tau'\cdot \eta$ is the representation given by $\tau'\otimes(\eta \circ \det)$.

\begin{thm}\label{Con}
Let $n$ be an integer greater or equal than $2$,  $\pi=\bigotimes_{x}' \pi_x$  an irreducible admissible representation of \textnormal{$\gb{GL}_n(\mathbb{A}_F)$} and  $\eta$ a continuous character of $\mathbb{A}^\times_F$ trivial on $F^\times$.

We suppose that, for a finite subset $S$ of places, $\pi$ satisfies the following properties:

\begin{enumerate}

\item The central character $\chi_\pi=\bigotimes_{x}' \chi_{\pi_x}$ of $\pi$ is invariant by the discrete subgroup $F^\times$ of $\mathbb{A}^\times_F$. 
\item For all $\pi'\in \mathcal{T}(S;\eta)$,  the formal series
\[L(s, \pi\times \pi') \text{ and }L(s, \Tilde \pi\times \Tilde\pi ')\]
are polynomials and they satisfy the functional equation
\[L(s,\pi \times \pi')=\varepsilon(s,\pi\times\pi',\psi)L(1-s,\Tilde \pi\times \Tilde\pi ').\]
\end{enumerate}
Then there exists an irreducible automorphic representation $\rho$ of \textnormal{$\gb{GL}_n(\mathbb{A}_F)$} such that, for each place $x\not \in S$ such that $\pi_x$ is unramified, $\rho_x$ is unramified and  $\pi_x\cong \rho_x$. Moreover, $\rho$ is cuspidal if $S=\emptyset$. 
\end{thm}
In order to prove this, first we introduce some notation. Secondly, we review some notions about Whittaker models. Thirdly, we prove a relation between Rankin-Selberg $L$-functions. Finally, we introduce further notations and we give a proof of converse theorem.

We denote by $\gb U_n$ the radical unipotent subgroup of the Borel $F$-subgroup $\gb B_n$  of upper triangular matrices of $\gb{GL}_n$. We also denote by ${\bf M}_{m,n}$ the algebraic group of $(m\times n)$-matrices over $F$. Finally, let us also write  
\[K=\prod_x K_x=\prod_x \gb{GL}_n(\mathcal O_x).\]
It is the maximal open compact subgroup of $\gb{GL}_n(\mathbb A_F)$, and $\gb{GL}_n(\mathbb A_F)$ is the restricted product of $\gb{GL}_n(F_x)$ with respect to the $K_x=\gb{GL}_n(\mathcal{O}_x)$.

If $\psi$ is a non-trivial character of either $F_x$ or $\mathbb A_F$, then we use $\psi$ to also denote the character of either $\gb U_n(F_x)$ or $\gb U_n(\mathbb A_F)$, that associates to $u=(u_{i,j})$ the complex number
\[\psi(u)=\sum_{i=1}^n\psi(u_{i,i+1}).\]

(\textit{Whittaker models and induced smooth representations of Whittaker type}). First, we recall that the induced (smooth) representations of Whittaker type are representations of the form
\[i_{\gb Q(F_x)}^{\gb{GL}_n(F_x)}(\rho_{1,x}|\det|^{u_{1,x}}\otimes \dots \otimes \rho_{m_x,x}|\det|^{u_{m_x,x}}),\]
where $\gb Q$ is a parabolic subgroup containing $\gb B_n$ associated to a partition $(r_{1,x},\dots,r_{m_x,x})$ of $n$, $\rho_{i,x}$ is an irreducible square-integrable representation of $\gb{GL}_{r_{i,x}}(F_x)$ for every $1\leq i \leq m_x$ and the $u_{i,x}$'s are real numbers satisfying $u_{1,x}\leq \dots\leq  u_{m_x,x}$. 
Every induced representation of Whittaker type  $\pi_x$ of ${\bf GL}_{n}(F_x)$ admits a $\psi_x$-Whittaker model for some non-trivial character $\psi_x\colon F_x\to \mathbb C^\times$, i.e. the space $\mathcal{W}(\pi,\psi_x)$ spanned by functions on $\gb{GL}_n(F_x)$ of the form 
\[g\mapsto \lambda_{\psi_x}(\pi_x(g)\xi_x),\]
where $\xi_x$ is a vector of $\pi_x$ and $\lambda_{\psi_x}\colon V \to \mathbb C$ is a functional such that
\[\lambda_{\psi_x}(\pi_x(u)v)=\psi_x(u)\lambda_{\psi_x}(v),\]
for all $u$ in $\gb U_n(F_x)$ and $v$ in $\pi_x$.  Once  the non-zero Whittaker functional is fixed, we denote such function by $W_{\xi_x}$. Note that each such function $W_{\xi_x}$ is right-invariant under some open subgroup of $\gb{GL}_n(F_x)$ and the collection  of these functions satisfies the following relation:
\[W_{\xi_x}(u_xg_x)=\psi_x(u_x)W_{\xi_x}(g_x), \text{ for every }g_x\in \gb{GL}_n(F_x), u_x\in \gb U_n(F_x).\]

Globally, let $\pi=\bigotimes_x' \pi_x$ be an admissible representation of $\gb{GL}_n(\mathbb A_F)$, where $\pi_x$ is induced of Whittaker type with fixed Whittaker functional.  We can choose  $K_x$-fixed vectors $\xi_x^\circ$, for $x$ outside some finite subset $T$ of $|F|$, such that $W_{\xi_x^\circ}\in \mathcal W(\pi_x,\psi_x)$  is invariant under right multiplication by the compact open subgroup $\gb{GL}_n
(\mathcal O_x)$ and it is equal to $1$ at the identity. Now, for every vector $\xi=( \xi_x)_{x\in |F|}$ of $\pi$, such that $\xi_x=\xi_x^\circ$ for almost all $x$, we consider the complex valued function on $\gb{GL}_n(\mathbb A_F)$ given by
\begin{align}
 W_\xi : g=(g_x)_x\mapsto \prod_{x}W_{\xi_x}(g_x). \label{eq:Wker}   
\end{align}
Each such $W_\xi$ is right-invariant  under by some  open compact  subgroup of $\gb{GL}_n(\mathbb A_F)$ and satisfies
\[W_\xi(ug)=\psi(u)W_\xi(g), \text{ for every }g\in \gb{GL}_n(\mathbb A_F), u\in \gb U_n(\mathbb A_F).\]
This function will be our main in ingredient  constructing a non-zero equivariant homomorphism to the space of automorphic forms.

(\textit{Twist}). Now, we briefly recall the definition of Rankin-Selberg $L$-functions of representation of Wittaker type and we prove a relation between them.

Let $\tau$ and $\tau'$ be induced  representations of Whittaker type of ${\bf GL}_n(F_x)$ and of ${\bf GL}_m(F_x)$, respectively. We define for any $W\in \mathcal W(\tau,\psi_x)$, $W'\in \mathcal W(\tau',\psi_x)$, and any compactly supported locally constant function $\Phi\colon F_x^n\to \mathbb C$, the following local integrals, which define rational functions in $\mathbb C(q_{F_x}^{-s})$ \cite[Theorem 2.7]{RSc}. In the case where $m<n$, for  $0\leq j\leq n-m-1$, we denote 
\begin{align*}\Psi_j(s;W,W')= \int_{\gb U_m(F_x)\backslash\gb{GL}_m(F_x)}\bigg(\int_{\gb M_{j,m}(F_x)}&W\begin{pmatrix}h &  & \\y & I_j & \\
      & & I_{n-m-j}\end{pmatrix}dy \bigg)   \cdot \\
    & W'(h)|\det(h)|^{s-(n-m)/2}dh.
\end{align*}
    In the case  where $m=n$, we put
    \[ \Psi(s;W,W',\Phi)=\int_{\gb U_m(F_x)\backslash\gb{GL}_m(F_x)}W(g)W'(g)\Phi((0,\dots,0,1)g)|\det(g)|^{s}dg.\]
These integrals form  $\mathbb C[q^s_{F_x},q_{F_x}^{-s}]$-fractional ideals $I(\tau,\tau')$ in the case where $n=m$, and $I_j(\tau,\tau')$, in the  case where $m<n$, for  $0\leq j\leq n-m-1$, in $\mathbb C(q_{F_x}^{-s})$. The  unique generator of these ideals has the form 
\[L(s,\tau\times \tau')=\frac{1}{P(q^{-s})}\]
with $P(X)\in \mathbb C[X]$ a polynomial with $P(0)=1$. This is the Rankin-Selberg $L$-function of $\tau \times \tau'$.

 \begin{pro} \label{Cogrmk}Let $k$ a locally compact field of positive characteristic,  $\eta:k^\times \to \mathbb{C}^\times$ a character and $(\tau,V), (\tau',V')$ two induced  representations of Whittaker type of $\gb{GL}_n(k)$ and $\gb{GL}_m(k)$, respectively. Then 
 \[L(s,\tau\times(\tau'\cdot \eta))=L(s,(\tau\cdot \eta)\times \tau').\]
 \end{pro} 
 \begin{proof} 
By definition (Section \ref{RS}), we notice that, after choosing a Whittaker functional $\Lambda:V\to  \mathbb C $ of $\tau$, we can compute the function $W_\xi \in \mathcal W(\tau\cdot \eta,\psi)$ as follows
 \[W_\xi(g)=\eta(\det(g))\Lambda(\tau(g)\xi)=\Lambda(\tau\cdot\eta(g)\xi).  \label{Witta}\]
 Now, let $\Lambda\colon V\to \mathbb C$ and $\Lambda'\colon V'\to \mathbb C$ be the respective Whittaker functionals of $\tau$ and $\tau'$, and $W_\xi \in \mathcal{W}(\tau,\psi)$ and $W'_{\xi'}\in \mathcal{W}(\tau'\cdot \eta,\psi)$. Then using the identity \ref{Witta} we get that 
 \[\Psi(s;W_\xi,W'_{\xi'})=\Psi(s;\Lambda(\tau(\cdot)\xi),\eta(\det(\cdot))\Lambda'(\tau'(\cdot)\xi')),\]
if $n=m$, and 
 \[\Psi_j(s;W_\xi,W'_{\xi'})=\Psi_j(s;\Lambda(\tau(\cdot)\xi),\eta(\det(\cdot))\Lambda'(\tau'(\cdot)\xi')),\]
  if $m<n$ and $0\leq j\leq n-m-1$. As these relations imply the equality of the ideals,  we have proved our desired relation. \end{proof}

(\textit{Further subgroups of $\gb{GL}_n$}). Finally, we introduce some notations. We fix a normal and proper curve  $X_F$ over $\mathbb F_q$  with field of fractions $F$. Denote by $\gb P_n\subset \gb{GL}_n$ the subgroup of matrices of the form 
\[\begin{pmatrix}*&\cdots &\cdots & *\\
\vdots & & &\vdots \\
*&\cdots&\cdots&* \\
0 &\cdots & 0& 1\end{pmatrix}.\]
For every closed subscheme $N$ of $X_F$ supported on $S$ with the ring of global sections denoted by $\mathcal O_N$, we consider the finite index subgroup $K_S'(N)$ of $K_S=\prod_{x\in S}\gb{GL}_n(\mathcal O_x)$ of matrices with image in $\gb{GL}_n(\mathcal O_N)$  of the form
\[\begin{pmatrix}*&\cdots &\cdots & *\\
\vdots & & &\vdots \\
*&\cdots&\cdots&* \\
0 &\cdots & 0& 1\end{pmatrix}.\]
We denote $\gb{GL}_n(\mathbb A_F)_S'(N)$ the open subgroup of $\gb{GL}_n(\mathbb A_F)$ given by the inverse image of $K_S'(N)$ under $\gb{GL}_n(\mathbb A_F)\to \prod_{x\in S}\gb{GL}_n(F_x)$.


Now we go  back to the proof of the converse theorem.
  \begin{proof}[Proof of Theorem \ref{Con}]
For every $x\in |F|$ such that $\pi_x$ unramified, we fix a vector $v_x\in V_x^{K_x}$.  For every $x$, let $\Xi_x$ be the representation of $\gb{GL}_n(F_x)$ that has  $\pi_x$ as its unique Langlands' quotient. Every $\Xi_x$ is of the form
 \[\Xi_x=i^{\gb{GL}_n(F_x)}_{\gb Q(F_x)}(\rho_{1,x}|\det|^{u_{1,x}}\otimes \dots \otimes \rho_{m_x,x}|\det|^{u_{m_x,x}}),\]
 where $\gb Q$ is a parabolic subgroup containing $\gb B_n$ associated to a partition $(r_{1,x},\dots,r_{m_x,x})$ of $n$, $\rho_{i,x}$ is an irreducible tempered representation of $\gb{GL}_{r_i,x}(F_x)$ for every $1\leq i \leq m_x$ and the $u_{i,x}$'s are real numbers satisfying $u_{1,x}> \dots>  u_{m_x,x}$. 

We can reduce the theorem to the case $\eta=1$. Indeed, by definition of Rankin-Selberg $L$-function  and using Proposition \ref{Cogrmk} we have
\begin{align*}
L(s,\pi_x\times(\pi'_x\cdot \eta_x))&=L(s,\Xi_x\times (\Xi_x'\cdot \eta_x))\\ &=L(s,(\Xi_x\cdot \eta_x)\times \Xi_x' )=L(s,(\pi_x\cdot \eta_x)\times \pi_x' ),    
\end{align*}
 we can apply  Theorem \ref{Con}, with trivial character, to $\pi\cdot \eta$. Therefore we have that there exists an automorphic representation $\Pi'$ such that $\Pi'_x\cong \pi_x\cdot \eta$ for $x\not \in S$ such that $\pi_x$ is unramified. Then $\Pi\coloneqq \Pi'\cdot \eta^{-1}$ is automorphic and satisfies that $\Pi_x\cong \pi_x$ for $x\not \in S$ such that $\pi_x$ is uramified. Therefore, from now on, we will assume that $\eta=1$.

Suppose that $S$ is not empty. For every  $x\not \in S$ for which $\pi_x$ is unramified, $\Xi_x$  must have a unique $K_x$-fixed vector $\xi_x^\circ$ which projects to the fixed $K_x$-fixed $v_x$ vector of $\pi_x$. From these choices,  we can consider for every $\xi=(\xi_x)_x$ such that $\xi_x=\xi_x^\circ$ for almost all $x\not \in S$, the global Whittaker  function $W_\xi$ \eqref{eq:Wker}. 

Now for every $x \in S$ such that $\pi_x$ is ramified, we can choose $\xi_x^\circ$ such that $(\xi_x^\circ)_{x\in S}$ is $K_S'(N)$-invariant for some subscheme $N$ of $X_F$, supported on $S$, and (\cite[Section 8 \& p. 203]{Cogcon}) 
 \[W_{\xi_x^\circ}(1)=1.\]
Thus, $\xi_x^\circ$ is  invariant under right multiplication by $\begin{pmatrix}h & 0\\
0 & 1
\end{pmatrix}$, with $h \in \gb{GL}_{n-1}(\mathcal O_x)$,  for  every $x\in S$.

Finally we consider as in \cite[Corollaire B.15]{L}, the well defined function on $\gb{GL}_n(\mathbb A_F)$
\[U_\xi(g)=\sum_{\gamma\in \gb U_n(F)\backslash \gb P_n(F)}W_\xi(\gamma g).  \label{introU}\]
Putting these together we are able to consider, for every $\xi^S=(\xi_x)_{x\not \in S}$ completed by $\xi=(\xi^S, (\xi_x^\circ)_{x\in S})$, the function   $U_{\xi^S}$ on $\gb{GL}_n(\mathbb A_F)$ defined by 
\[U_{\xi^S}(g)=U_\xi(g'),\]
if $g$ can be written as $g=\gamma g'$ with $\gamma\in \gb{GL}_n(F)$ and $g'\in \gb{GL}_n(\mathbb A_F)_S'(N)$ 
and, if not, by
\[U_{\xi^S}(g)=0.\]
The map $\xi^S\mapsto U_{\xi^S}$ defines a non-zero \cite[Lemma 6.3]{Cogcon} equivariant homomorphism  of the smooth admissible representation $\Xi^S=\bigotimes_{x\not \in S}' \Xi_x$ of $\prod_{x\not \in S}'\gb{GL}_n(F_x)$ to the space of functions on $\gb{GL}_n(F) \backslash \gb{GL}_n(\mathbb A_F)$ that are  invariant under right multiplication by open compact subgroups of $\gb{GL}_n(\mathbb A_F)$ \cite[p. 237]{L}.  The action of the center $\gb Z_n(\mathbb A_F)$ of $\gb{GL}_n(\mathbb A_F)$ on the span  of these  functions is according to the central character  $\chi_\pi$ of $\pi$. 

Since $\Xi^S$ has $\Pi^S=\bigotimes_{x\not \in S}'\Pi_x$ as its unique irreducible quotient, if we take a vector $\xi^S$ which has a non-zero projection to $\Pi^S$, then $\xi^S$ is a cyclic generator of $\Xi^S$. Thus the representation $V$ of $\gb{GL}_n(\mathbb A_F)$ generated by the space of $U_{\xi^S}$ is admissible \cite[Section 5]{CorBJ} and cyclic, generated by some element $f_0$. Let  $U$ be a maximal $\gb{GL}_n(\mathbb A_F)$-invariant subspace of $V$ not containing $f_0$.  Then $\Pi'=V/U$ is a non-zero subquotient of the space of automorphic forms; $\Pi'$ is automorphic and at every place $x\not \in S$ where $\pi_x$ is unramified, its Satake parameter equal that to the one of $\pi_x$ \cite[Theorem A]{Cogcon}.

In the case where $S$ is empty, we just consider $(\xi \mapsto U_\xi)$. As $U_\xi$ is cuspidal \cite[Proposition 12.3]{GL3II}, we can conclude as before.
\end{proof}

 Next, the tool that will allow us to establish the desired properties of the $L$-functions is provided by the Langlands-Shahidi method.

\section{The Langland-Shahidi method} 
The Langlands-Shahidi method in positive characteristic \cite{L18}, studies  $\gamma$-factors, $\varepsilon$-factors and $L$-functions for generic representations associated to irreducible constituents of the adjoint representation of ${}^LM$ on ${}^L\mathfrak{n}$, where $\mathfrak{n}$ is the Lie algebra of ${}^LN$.

Let $\gb P=\gb M\gb N$ be a maximal parabolic subgroup of $\gb{SO}_{2m+2n}^*$ containing $\gb B$. Given the structure given in Section \ref{section:structure}, we have an isomorphism $\gb M\cong \gb{GL}_m \times \gb{SO}^*_{2n}$, where $m$ and $n$ are greater than  $1$.
In this case, we  obtain the following decomposition of the adjoint representation
\[r=r_1\oplus r_2,  \label{adjr}\]
where $r_1=\rho_m\otimes \Tilde \rho^*_{2n}$  and $r_2=\wedge^2\rho_{m}\otimes 1_{\gb{SO}_{2n}^*}$ \cite[p. 565]{Shram}. Here $\rho_m$ is the standard representation of ${}^L\gb{GL}_m(\mathbb C)$, $\rho_{2n}^*$ the  representation of ${}^L\gb{SO}^*_{2n}(\mathbb C)$ constructed in Section 3.2 eq. \eqref{L-hom} and $1_{\gb{SO}_{2n}^*}$ is the trivial representation of ${}^L\gb{SO}_{2n}^*(\mathbb C)$. Thus we obtain two instances of Langlands-Shahidi $L$-functions, $\varepsilon$-factors and $\gamma$-factors.
\subsection{Local construction}\label{subsec:LocalLS}
Suppose $F$ is a locally compact field of positive characteristic. Recall that  the cardinality of the residue field is denoted by $q_F$. Let $E$ be a separable quadratic extension of $F$ contained in a fix separable closure $F_s$, with Galois group $\operatorname{Gal}(E/F)=\{1,\sigma\}$.

Let  $\psi\colon F\to \mathbb C^\times$ be a smooth non-trivial character, $\pi$  a generic representation of $\operatorname{\textbf{SO}}_{2n}^*(F)$ and $\tau$ a generic representation of $\operatorname{\textbf{GL}}_m(F)$. Then $\tau \otimes \Tilde \pi$ ($\Tilde \pi$ is the contragredient of $\pi$)  is a generic representation of $\operatorname{\textbf{M}}(F)$.
\begin{itemize}
    \item For $r_1$, we denote the corresponding local factors by 
\[\gamma(s,\pi\times\tau,\psi),\quad   \varepsilon(s,\pi\times\tau,\psi) \quad \& \quad L(s,\pi\times\tau).\]  
Observe that if we allow $E=F\times F$, we obtain the local construction for the split groups $\gb{SO}_{2n}(F)$, already studied in \cite{L09}.
\item For $r_2$, we denote the corresponding local factors by
\[\gamma(s,\tau,\wedge^2\rho_m,\psi)\quad   \varepsilon(s,\tau,\wedge^2\rho_m,\psi) \quad \& \quad L(s,\tau,\wedge^2\rho_m).\]
These factors have already been studied in detail in \cite{HLsym}
\end{itemize}    
Among the properties of these factors, we would like to highlight the multiplicativity property \cite[Section 5]{L15}. 

(\textit{$r_1$ case}). We have the following two versions. Let  $\gb M_1=\gb{GL}_m\times \gb{GL}_{n_b}\times \cdots \times \gb{GL}_{n_1}\times \gb{SO}^*_{2n_0}\subset \gb M$, and suppose that $\pi$ is the generic subquotient of 
\[i_{P_1}^{\gb{SO}_{2n}^*(F)}(\pi_b \otimes \cdots  \otimes\pi_1  \otimes \pi_0),\]
where $\gb P_1=\gb M_1\gb N_1$ is the parabolic subgroup of $\gb{SO}_{2n}^*$, containing $\gb B$, $\pi_i$ is a generic representation of $\gb{GL}_{n_i}(F)$ for $1\leq i\leq b$ and $\pi_0$  is a generic representation of $\gb{SO}^*_{2n_0}(F)$. Then the multiplicative property \cite[Section 5]{L15} gives us
\begin{align}\label{eq:mult}
\gamma(s, \pi \times \tau,\psi)=\gamma(s,\pi_0 \times \tau,\psi)\prod_{i=1}^{b}\gamma(s,\pi_i \times \tau,\psi)\gamma(s,\Tilde \pi_i\times \tau,\psi),    
\end{align}
where $\gamma(s,\pi_i \times \tau,\psi)$ is the Rankin-Selberg $\gamma$-function (Section \ref{RS}). For the other case, let $\gb M_2=\gb{GL}_{m_b}\times \cdots \times\gb{GL}_{m_1}\times\gb{SO}^*_{2n}\subset\gb M$ and suppose that $\tau$ is the generic subquotient of 
\[i_{Q}^{\gb{GL}_{m}(F)}(\tau_{b}\otimes \cdots\otimes \tau_{1}),\]
    where $\gb Q=\gb M_2\gb N_2$ is the parabolic subgroup  of $\gb{GL}_m$ containing the upper triangular matrices, $\tau_i$ is a generic representation of $\gb{GL}_{m_i}(F)$ for $1\leq i\leq b$. Then 
\begin{align}
\gamma(s,\pi \times \tau,\psi)=\prod_{i=1}^b\gamma(s,\pi\times \tau_i,\psi)\label{eq:mult1}.    
\end{align}
Before continuing to the $r_2$ case, we make more explicit the principal series case. Let $(\chi_1,...,\chi_{n-1},\chi)$ be a character of the maximal subtorus $\gb T(F)$ of $\operatorname{SO}_{2n}^*=\gb{SO}(q_{E,n})(F)$, where $\chi_i$ is a character of $F^\times$ for each $1\leq i \leq n-1$ and $\chi$ is a character of  $E^1$. Then, if $\pi$ is the generic subquotient of \[i^{SO_{2n}^*}_{B}(\chi_1\otimes \cdots\otimes\chi_{n-1}\otimes\chi)\] and $\xi$ a character of $F^\times$, the multiplicativity formula gives us
\[\gamma(s, \pi\times \xi, \psi)=\gamma(s,\chi\times\xi,\psi)\prod_{i=1}^{n-1}\gamma(s,\chi_i\xi,\psi)\gamma(s,\chi_i^{-1} \xi ,\psi)\label{Cprin}\]
where the $\gamma(s,\chi_i\xi,\psi)$ are Tate factors.

Let us study the rank one case. First write $\psi_E=\psi\circ \operatorname{Tr}_{E/F}$ and let $\lambda(E/F,\psi)$ be the Langlands constant \cite[Section 30.4]{BHGL2}. Now let  us recall that we constructed the simply connected cover of $\gb{SO}(q_{E,2})$ (Section \ref{rnk1}):
\[ \operatorname{Res}_{E/F}\gb{SL}_2\to \gb{SO}(q_{E,2}).\]
This morphism restricts to
\[\operatorname{diag}(t,t^{-1})\mapsto [(x_1,x_2,x) \mapsto (N_{E/F}(t)x_1,N_{E/F}(t)^{-1}x_2,t\sigma^{-1}(t)x)].\]
 Thus we have the following \cite[Proposition 1.3]{L15}. 
\begin{pro}\label{2.9:gamma2}
Let $(\chi,\xi)$ be a smooth character  of $\gb T(F)$, and  $\mu$ the character of $E^\times$ defined by $[t\mapsto (\chi\circ N_{E/F})(t)\cdot \xi(t\sigma^{-1}(t))]$.
Then \[\gamma(s,\chi\times \xi,\psi)=\lambda(E/F,\psi)\gamma(s,\mu,\psi_{E}).\]
\end{pro}

(\textit{$r_2$ case}). The multiplicative property in the case of $r_2$ has the following form. Let $\gb M_1=\gb{GL}_{m_b}\times \cdots \times\gb{GL}_{m_1}\times\gb{SO}^*_{2n}\subset\gb M$ and suppose that $\tau$ is the generic subquotient of 
\[i_{P}^{\gb{GL}_{m}(F)}(\tau_{b}\otimes \cdots\otimes \tau_{1}),\]
  where $\tau_i$ is a generic representation of $\gb{GL}_{m_i}(F)$ for $1\leq i\leq b$. Then 
\[\gamma(s,\tau,\wedge^2\rho_{m},\psi)=\prod_{i=1}^b\gamma(s,\tau_i,\wedge^2\rho_{m_i},\psi)\prod_{i<j}\gamma(s,\tau_i\times \tau_j,\psi).\]

Another important property are the following two stability results for the $\gamma$-functions. First, for any smooth character $\eta\colon F^\times \to \mathbb{C}^\times$ and any smooth representation $\tau$  of ${\bf GL}_m(F)$, we write  
\[\tau \cdot \eta =\tau \otimes (\eta \circ \det). \]

\begin{lem}\cite[Main Lemma 1]{Shtwist} \label{gamtwist}  Let $\pi$ be a generic representation of $\gb{SO}^*_{2n}(F)$ and $\tau$ a generic representation of $\gb{GL}_m(F)$. Then there exists a character $\chi$ of $F^\times$ so that $\gamma(s,\pi \times (\tau \cdot \chi),\psi)$ is a monomial in $q_F^{-s}$,
for $1\leq i\leq m$. Moreover $\chi$ can be replaced by any character of $F^\times$ whose conductor is larger than that of $\chi$.
\end{lem}

The other  important  stability result is the following.
\begin{thm} \cite[Corollary 6.5]{GL} \label{sthm}
Let $\pi_1$ and $\pi_2$ be two irreducible generic representations of $\gb{SO}^*_{2n}(F)$ having the same central character, and let $\tau$ be an irreducible generic representation of $\gb{GL}_m(F)$. Then for a sufficiently highly ramified character $\chi$ of $F^\times$, we have 
\[ \gamma(s,\pi_1 \times (\tau \cdot \chi),\psi) = \gamma(s,\pi_2 \times (\tau \cdot \chi),\psi).\]
\end{thm}
\subsection{Global construction} 
Suppose now that $F$ is a function field in one
variable over a finite field. Let $\psi=\bigotimes_{x}'\psi_x: \mathbb A_F/F \to \mathbb C^\times$ be a continuous non-trivial character, $\pi=\bigotimes'_x \pi_x$  a globally generic cuspidal  autmorphic  representation of $\gb{SO}_{2n}^*(\mathbb A_F)$ and $\tau=\bigotimes_x'\tau_x$ a cuspidal automorphic   representation of $\gb{GL}_m(\mathbb A_F)$. Then $\tau \otimes \Tilde \pi$  is a globally generic cuspidal automorphic   representation of $\operatorname{\textbf{M}}(\mathbb A_F)$. If $S$ is a finite subset of $|F|$, such that $\pi_x$, $\tau_x$ and $\psi_x$ are unramified, we denote the partial $L$-function as follows.

(\textit{$r_1$ case}). For $r_1$, we denote the partial $L$-functions by
\[L^S(s,\pi\times\tau)=\prod_{x\not \in S} L(s,\pi_x\times\tau_x).\]
and the completed $L$-functions and $\varepsilon$-factors by
\[L(s,\pi\times\tau)=\prod_{x\in |F|}L(s,\pi_x\times\tau_x) \text{ and } \varepsilon(s,\pi\times\tau)=\prod_{x\in |F|}\varepsilon(s,\pi_x\times\tau_x,\psi_x).\]
They satisfy the functional equation \cite[Section 5.5]{L15}
\begin{align}\label{eq:funeqLS}
L(s,\pi\times\tau)=\varepsilon(s,\pi\times\tau) L(1-s,\Tilde\pi\times\Tilde\tau).    
\end{align}

(\textit{$r_2$ case}). Similarly, for $r_2$, we denote the partial $L$-functions by\[L^S(s,\tau,\wedge^2\rho_m)=\prod_{x\not \in S} L(s,\tau_x,\wedge^2\rho_m),\]
and the  completed $L$-functions and $\varepsilon$-factors by
\[L(s,\tau,\wedge^2\rho_m) =\prod_{x}L(s,\tau_x,\wedge^2\rho_m)  \text{ and } \varepsilon(s,\tau,\wedge^2\rho_m)=\prod_{x} \varepsilon(s,\tau_x,\wedge^2\rho_m,\psi_x) .\]
They also satisfy the functional equation \cite[Section 5.5]{L15}
\[L(s,\tau,\wedge^2\rho_m)=\varepsilon(s,\tau,\wedge^2\rho_m) L(1-s,\Tilde \tau,\wedge^2\rho_m).\]

We now go back to the generic functoriality.

\section{Generic functoriality for $\gb{SO}_{2n}^*$}\label{section:genericfunc}
 Let $\pi=\bigotimes_x' \pi_x$ be a globally generic cuspidal automorphic representation of ${\bf SO}_{2n}^*(\mathbb A_F)$. We apply Langlands-Shahidi method to $\pi$ and the candidate lift construction made in Section \ref{subsec:description}. In this section, we will focus in the case $x$ is an inert place. The split case version of the result in this section are obtained as in \cite{L09}.  Now, we start with case when $\pi_x$ is unramified.
\subsection{Local lift} \label{subsec:locallift}
\begin{pro}
 Let $\pi_x$ be an unramified generic irreducible representation of $\gb{SO}_{2n}^*(F_x)$ and $\Pi_x$ a candidate lift as in Section \ref{subsec:candidate}. Then for  a generic irreducible representation $\tau_x$ of $\gb{GL}_{m}(F_x)$ we have the following
\begin{equation}
 \label{unr}
 \begin{gathered}
 L(s,\pi_x \times \tau_x  )=L(s,\Pi_x\times\tau_x ),\\
 \varepsilon(s,\pi_x\times \tau_x, \psi_x)=\varepsilon(s,\Pi_x\times\tau_x,\psi_x).
 \end{gathered}
 \end{equation}
 \end{pro}
 \begin{proof}
We start with the setup of the definition of local factors \cite[Appendix A]{dCHL}. Let 
 \[i_P^{SO^*_{2n}}(\pi_{1,x}|\det|^{r_1},\dots,\pi_{b,x}|\det|^{r_b},\pi_{0,x})\]
 be the induced representation such that $\pi_x$ is its Langlands quotient, and where $0<r_1< \cdots < r_b$, ${\bf P}$ is a parabolic subgroup of ${\bf SO}_{2n}^*$ containing $\gb B$,  $\pi_{i,x}$ is an irreducible tempered representation of $\gb{GL}_{n_i}(F_x)$ for $1\leq i\leq b$, $\pi_{0,x}$  is an irreducible tempered representation of $\gb{SO}_{2n_0}^*(F_x)$. Let
 \[i_Q^{GL_m}(\tau_{1,x}|\det|^{t_1},\dots,\tau_{d,x}|\det|^{t_d})\]
 be the induced representation such that $\tau_x$ is its Langlands quotient, and where $0<t_1< \cdots < t_d$, $\gb Q$ is a parabolic subgroup containing the Borel subgroup of $\gb{GL}_m$ consisting of upper triangular matrices, the $\tau_i$'s are generic unitary tempered representation of $\gb{GL}_{m_i}(F_x)$. By definition we have
 \begin{equation}\label{eq:Local-L}
     L(s,\pi_x\times \tau_x)=\prod_{j=1}^d L(s+t_j,\pi_{0,x}\times \tau_{j,x})\prod_{i=1}^{b}\prod_{j=1}^d L(s+t_j+r_i,\pi_{i,x}\times \tau_{j,x})L(s-r_i+t_j,\Tilde \pi_{i,x}\times \tau_{j,x}).
 \end{equation}
With the setup ready, we first study every factor individually.  Since  $\pi_x$ is unramified, we have that $\pi_{i,x}$ is unramified. 
Let $\Pi_{0,x}$  be a lift as in Section \ref{subsec:candidate} of the unramified representation $\pi_{0,x}$.
From multiplicativity we observe that, for $1\leq j \leq d$,
 \begin{align*}
\gamma(s,\pi_{0,x}\times \tau_{j,x},\psi_x)&=\gamma(s,\Pi_{0,x}\times \tau_{j,x},\psi_x).
 \end{align*}
 Since the Satake parameters of $\pi_{0,x}$ have absolute value equal to $1$ , $\Pi_{0,x}$ is tempered \cite{dCHL}. Thus, for every $0\leq i\leq b$ and $1\leq j\leq d$, \begin{align}\label{eq:gammaepsilonL-unr}
\varepsilon(s,\pi_{0,x}\times \tau_{j,x},\psi_x) \frac{L(1-s,\Tilde \pi_{0,x}\times \Tilde \tau_{j,x})}{L(s,\pi_{0,x}\times \tau_{j,x})}&=\varepsilon(s,\Pi_{0,x}\times \tau_{j,x},\psi_x) \frac{L(1-s,\Tilde \Pi_{0,x}\times \Tilde \tau_{j,x})}{L(s,\Pi_{0,x}\times \tau_{j,x})}.
 \end{align}
 From the tempered $L$-function conjecture \cite{dCHL}, we have that $L(s,\pi_{0,x}\times \tau_{j,x})$ and $L(s,\Pi_{0,x}\times \tau_{j,x})$ are holomorphic on $\operatorname{Re}(s)>0$. Furthermore, the regions where these $L$-functions have poles do not intersect. Therefore, there are no cancellations involving the numerator and denominator and thus
 \[L(s,\pi_{0,x}\times \tau_{j,x})=L(s,\Pi_{0,x}\times \tau_{j,x}),\]
 for every  $1\leq j\leq d$. Using this on the right hand side of  \eqref{eq:Local-L}, we obtain that   \[L(s,\pi_x\times \tau_x)=L(s,\Pi_{0,x}\times \tau_x)\prod_{i=1}^bL(s+r_i,\pi_{i,x}\times \tau_x)L(s-r_i,\Tilde \pi_{i,x}\times \tau_x).\] We note that the unramified component of 
 \[i_{Q'}^{GL_{2n}}(\pi_{1,x}|\det|^{r_1}\otimes\cdots\otimes\pi_{b,x}|\det|^{r_b}\otimes\Pi_{0,x}\otimes\Tilde \pi_{b,x}|\det|^{-r_b}\otimes \cdots \otimes\Tilde \pi_{1,x}|\det|^{r_1} ),\] where ${\bf Q}'$ is the parabolic subgroup of ${\bf GL}_{2n}$ containing ${\bf B}_{2n}$ associated to the partition $(n_1,\dots,n_b,2n_0,n_b,\dots,n_1)$ of $2n$, is $\Pi_x$. Then the right hand side of the last expression is equal to $L(s,\Pi_{x}\times \tau_x)$. Therefore, we obtain the desired relation between the $L$-functions.
 
 Similarly, using the analogous relation between $L$-function of the contragredient representations and using \eqref{eq:gammaepsilonL-unr}, we can obtain the desired relations for the $\varepsilon$-factors.
 \end{proof}
 Now, we analyse the case when $\pi_x$ is ramified. In this case, we twist our representation in order to reduce to a $\gamma$-factor relation.

\begin{pro}Let $\pi_x$ be an irreducible generic representation of $\gb{SO}_{2n}^*(F_x)$ and $\Pi_x$ a candidate lift as in Section \ref{subsec:candidate}. Then for any sufficiently ramified enough character $\eta_x$ of $F^\times_x$, we have that
\begin{equation}
\label{ram}
\begin{gathered}
L(s, \pi_x\times(\tau_x\cdot\eta_x) )=L(s,\Pi_x\times (\tau_x\cdot \eta_x)), \\
\varepsilon(s,\pi_x\times(\tau_x\cdot\eta_x),\psi_x)=\varepsilon(s,\Pi_x\times (\tau_x\cdot\eta_x),\psi_x),
\end{gathered}
\end{equation}
for every unramified irreducible representation $\tau_x$  of $\gb{GL}_{m}(F_x)$.
\end{pro}
\begin{proof}As before, we start with the setup of the definition of local factors. Let 
 \[i_P^{SO^*_{2n}}(\pi_{1,x}|\det|^{r_1}\otimes\dots\otimes\pi_{b,x}|\det|^{r_b}\otimes\pi_{0,x})\]
 be the induced representation such that $\pi_x$ is its Langlands quotient, and where $0<r_1< \cdots < r_b$, ${\bf P}$ is a parabolic subgroup of ${\bf SO}_{2n}^*$ containing $\gb B$,  $\pi_{i,x}$ is an irreducible tempered representation of $\gb{GL}_{n_i}(F_x)$ for $1\leq i\leq b$, $\pi_{0,x}$  is an irreducible tempered representation of $\gb{SO}_{2n_0}^*(F_x)$. Let
 \[i_Q^{GL_m}(\tau_{1,x}|\det|^{t_1}\otimes\dots\otimes\tau_{d,x}|\det|^{t_d})\]
 be the induced representation such that $\tau_x$ is its Langlands quotient, and where $0<t_1< \cdots < t_d$, $\gb Q$ is a parabolic subgroup containing the Borel subgroup of $\gb{GL}_m$ consisting of upper triangular matrices, the $\tau_{i,x}$'s are generic unitary tempered representation of $\gb{GL}_{m_i}(F_x)$. 
Finally, let 
\[i_{Q'}^{GL_{2n}}(\Pi_{1,x}|\det|^{s_1}\otimes\dots\otimes\Pi_{l,x}|\det|^{s_l})\]
 be the induced representation such that $\Pi_x$ is its Langlands quotient, and where $0<s_1< \cdots < s_d$, $\gb Q'$ is a parabolic subgroup containing the Borel subgroup of $\gb{GL}_{2n}$ consisting of upper triangular matrices, the $\Pi_{i,x}$'s are generic unitary tempered representation of $\gb{GL}_{l_i}(F_x)$.

Now, making $\eta_x$ sufficiently ramified to obtain (Lemma \ref{gamtwist}), we obtain that
\[L(s, \pi_{i,x}\times(\tau_{j,x}\cdot\eta_x))\equiv 1 \equiv L(s,\Pi_{k,x}\times (\tau_{j,x}\cdot \eta_x)),  \label{L=1}\]
for every $0\leq i\leq b$, $1\leq j\leq d$ and $1\leq k \leq l$.
By definition of $\varepsilon$-factors, this implies that
\begin{align*}
    \varepsilon(\pi_x\times(\tau_x\cdot\eta_x),\psi_x) &= \gamma(\pi_x\times(\tau_x\cdot\eta_x),\psi_x), \\
    \varepsilon(s,\Pi_x\times (\tau_x\cdot\eta_x),\psi_x)&= \gamma( s,\Pi_x\times (\tau_x\cdot\eta_x),\psi_x).
\end{align*}
Thus, we are left to prove the corresponding identity for the $\gamma$-factors.

Now, as $\pi_x'$ is generic  we can use the stability of the gamma factors (Theorem \ref{sthm}). By this result, if we make $\eta_x$ ramified enough, the following identity also holds 
\[\gamma(s, \pi_x\times(\tau_x\cdot\eta_x),\psi_x)= \gamma(s,  \pi_x'\times(\tau_x\cdot\eta_x),\psi_x ).\]
On the other hand, using the explicit description of the $\gamma$-factors in the principal series case in Section \ref{subsec:LocalLS}, we have 
\begin{align*}
 \gamma(s, \pi_x'\times\eta_x,\psi_x)&=\gamma(s,\chi_{n,x} \times \eta_x,\psi_x)\prod_{i=1}^{n-1}\gamma(s,\chi_{i,x}\eta_x
 ,\psi)\gamma(s,\chi_{i,x}^{-1}\eta_x ,\psi_x) \\
&=\gamma(s,\Pi_{\mu_{n,x}}\times \eta_x,\psi_x)\prod_{i=1}^{n-1}\gamma(s,\chi_{i,x}\eta_x
 ,\psi)\gamma(s,\chi_{i,x}^{-1}\eta_x ,\psi_x)\\
&=\gamma(s,\Pi_x\times  \eta_x,\psi_x)  \label{gamram} 
\end{align*} 

Now, since $\tau_x$ is unramified, it is a subquotient of an induced  representation of the form
\[i_{\gb B_m(F_x)}^{\gb{GL}_m(F_x)}(|\cdot|^{b_1}\otimes \cdots\otimes |\cdot|^{b_m}),\] where $b_i\in \mathbb C$. Using the multiplicativity of  the $\gamma$-factors, we have  
\begin{align*}
\gamma(s, \pi_x\times(\tau_x\cdot\eta_x),\psi_x )=\prod_{i=1}^{m}\gamma(s-b_i, \pi_x\times\eta_x)
\end{align*}
and
\begin{align*}
\gamma(s,\Pi_x\times (\tau_x\cdot\eta_x))&=\prod_{i=1}^{m} \gamma(s,\Pi_x\times (|\cdot|^{b_i}\cdot \eta_x)) \\
&=\prod_{i=1}^{m}\gamma(s-b_i,\Pi_x\times \eta_x).
\end{align*}
Comparing these two, we obtain the desired identity.
\end{proof}

\subsection{Global lift}
Using the equalities \eqref{unr} twisted by any sufficiently ramified enough character and \eqref{ram}, we have:
\begin{cor}\label{ShRS} Let $\pi=\bigotimes_x' \pi_x$ be a globally generic cuspidal automorphic representation of $\gb{SO}_{2n}^*(\mathbb{A}_F)$, unramified outside of a non-empty $S\subset|F|$ and let $\Pi$ a candidate lift of $\pi$ as in Section \ref{subsec:candidate}. Then, for a character $\eta=\bigotimes'_x\eta_x$, sufficiently ramified in $x \in S$, (so as to satisfy \eqref{ram}), we have
\begin{equation}   \begin{gathered}
L(s,\pi\times\tau)=L(s,\Pi\times \tau),\\
\varepsilon(s,\pi\times\tau)
=\varepsilon(s,\Pi\times \tau),
\end{gathered}
\end{equation}
for every $\tau\in \mathcal{T}(S;\eta)$ (as in Section \ref{subsec:converse}).
\end{cor}
\label{aplicon} We know that a lift $\Pi$ of $\pi$ is irreducible and admissible and that its central character is trivial on $F^\times$ \eqref{central}, but we do not  necessarily have that it is automorphic. For  that we use the converse theorem (Theorem \ref{Con}). Thus, we need to make sure that $L(s,\pi \times \tau)$ is a polynomial for $\tau \in \mathcal{T}(S;\eta)$, for some character $\eta$ of $\mathbb A_F^\times$  trivial on $F^\times$. In order to obtain that we are going to study the local Normalized Intertwining Operator.

\subsection{Local Normalized Intertwining Operator}\label{subsec:normalizedOp}
First, let us review the definition of the Local Normalized Intertwining Operator. Thus, assume that $F$ is a locally compact field of positive characteristic.

  Let $\gb P=\gb M\gb N$ be a maximal parabolic subgroup of $\gb{SO}_{2m+2n}^*$ containing $\gb B$, such that $\gb M\cong \gb{GL}_m \times \gb{SO}^*_{2n}$, associated to $\Delta\setminus \{\alpha\}$ and $\Tilde w_0\in \gb G(F)$ a representative of $w_0=w_{l,\gb G}w_{l,\gb M}\in\,W^{\gb G}$, where $w_{l,{\bf G}}$ and $w_{l,{\bf M}}$ are the longest element of the Weyl group of ${\bf G}$ and of ${\bf M}$, respectively.  For $\sigma= \tau\otimes\Tilde\pi$  a generic  representation of $\gb M(F)$, where $\tau$ is a representation of ${\bf GL}_m(F)$ and $\pi$ is a representation of ${\bf SO}_{2n}^*(F)$, we define
\[r(s,\sigma)=\frac{L(s,\pi\times \tau)L(2s,\tau,\wedge^2\rho_m)}{L(1+s,\pi\times \tau)L(1+2s, \tau,\wedge^2\rho_m)\varepsilon(s,\pi\times \tau,\psi)\varepsilon(2s,\tau,\wedge^2\rho_m,\psi)}, \]
and the normalized intertwining operator $N(s,\sigma, \Tilde w_0)$ is defined to be such that
\[A(s,\sigma,\Tilde w_0)=r(s,\sigma)N(s,\sigma,\Tilde w_0)\colon i^G_{P_{\theta}}(s,\sigma)\to  i^G_{P_{\theta'}}(\Tilde w_0(s),\Tilde w_0(\sigma)) \label{Kint}\]
 as a rational operator in $s$, where $A(s,\sigma,\Tilde{w}_0)$ is the operator is given by 
 \[A(s,\sigma,\Tilde w_0)f(g)=\int_{w_0U_\theta w_0^{-1}\cap U_{\theta'}\backslash U_{\theta'}}f(\Tilde w^{-1}_0u'g)du'.\]

 The holomorphicity of the local normalized intertwining operator relies on two local properties: the first one  is the so-called standard module conjecture. In our case, it states that every generic smooth irreducible representation of $\gb{SO}_{2n}^*(F)$ is the full induced representation
\begin{equation}\label{eq:SM}
    i_{P}^{\gb{SO}^*_{2n}(F)}(\pi_{b}|\det|^{r_b}\otimes \cdots\otimes \pi_{1}|\det|^{r_1}\otimes \pi_{0}),
\end{equation}
where $0<r_1\leq \cdots \leq r_b$, with $r_b<1$, $\pi_i$ is an irreducible tempered representation of $\gb{GL}_{n_i}(F)$ for $1\leq i\leq b$ and $\pi_0$  is an irreducible tempered representation of $\gb{SO}_{2n_0}^*(F)$. In order words, the standard modules of generic irreducible smooth representations are irreducible. This conjecture have been proven for a general quasi-split groups in characteristic zero \cite{HM} and in  positive characteristic \cite{dCHL}.
 
The second ingredient is that the real numbers $r_i's$ appearing in \eqref{eq:SM}  are less than $1$. To obtain that, we will use a local-global result (Proposition \ref{pro:Lkim}) and then we will follow the Kim's arguments in \cite[Section 3]{Kim2000} to obtain the bound. 

Using these two properties and following \cite[Proposition 3.4]{Kim2000}, we have the following.
\begin{pro} \label{pro:KimA} Let $\pi$ be generic  representation of $\gb{SO}_{2n}^*(F)$ such that it is the full induced representation
\[i_{P}^{\gb{SO}^*_{2n}(F)}(\pi_{b}|\det|^{r_b}\otimes \cdots\otimes \pi_{1}|\det|^{r_1}\otimes \pi_{0}),\]
where $0<r_1\leq \cdots \leq r_b$, with $r_b<1$, $\pi_i$ is an irreducible tempered representation of $\gb{GL}_{n_i}(F)$ for $1\leq i\leq b$ and $\pi_0$  is an irreducible tempered representation of $\gb{SO}_{2n_0}^*(F)$. Then $N(s,\tau \otimes \Tilde \pi,\Tilde{w}_0)$ is holomorphic and non-zero on $\operatorname{Re}(s)\geq 1/2$ for every generic unitary representation $\tau$ of $\gb{GL}_m(F)$.
\end{pro}
\begin{proof} We can write $\tau$ as the full induced representation  \cite[Section 7]{Tad}
 \[i_Q^{\operatorname{GL}_m}(\xi_1|\det|^{t_1}\otimes \cdots \otimes \xi_d|\det|^{t_d}\otimes \xi_{d+1} \otimes\xi_d|\det|^{-t_d} \otimes \cdots \otimes
\xi_1|\det|^{-t_1}), \]
where $\gb Q$ is a parabolic subgroup containing the Borel subgroup of $\gb{GL}_m$ consisting of upper triangular matrices, the $\xi_i$'s are tempered representations of $\gb{GL}_{m_i}(F)$ and $0<t_1\leq \cdots \leq t_d<1/2$. 

Combining the description for $\pi$ and $\tau$ as induced representations, we obtain that $\tau \otimes \Tilde \pi$ is full induced from quasi-tempered datum. This allows us to use multiplicativity of the normalized intertwining operators (See \cite[Proposition 4.6]{HK}), in order to reduce to the following rank one cases  $\gb{GL}_k \times \gb{GL}_l\subset \gb{GL}_{l+k}$, $\gb{SO}_{2l}^*\times \gb{GL}_k\subset \gb{SO}_{2(l+k)}^*$ and $\gb{GL}_{l-1}\subset \gb{SO}_{2l}^*$ ($l\geq 3$):
\begin{enumerate}
    \item For the case $\gb{GL}_k \times \gb{GL}_l\subset \gb{GL}_{l+k}$, we obtain from $\operatorname{Re}(s\pm r_i\pm t_j )>-1$, for $\operatorname{Re}(s) \geq 1/2$ the condition, thanks to \cite[Proposition I.10]{MWspectre}.
    \item For the case $\gb{SO}_{2l}^*\times \gb{GL}_k\subset \gb{SO}_{2(l+k)}^*$, we note that $\operatorname{Re}(s\pm t_d)\geq 0$ for $\operatorname{Re}(s)\geq 1/2$. As $\pi_0$ is tempered, we get our condition. 
    \item Finally for the case $\gb{GL}_{l-1}\subset \gb{SO}_{2l}^*$ ($l\geq 3$), we conclude as in \cite[Lemma 3.3, Proposition 3.4]{Kim1999} to conclude.
\end{enumerate}
\end{proof}  
\subsection{Generic functoriality for $\gb{SO}_{2n}^*$} Now let us go back to the global situation and we are going to check the last points to obtain the generic functoriality.

(\emph{Polynomial condition}). To prove the polynomial condition, we need the following holomorphicity result of the $L$-functions. The result is inspired from \cite[Section 2]{KimSha} and \cite[Section 4]{L18}.
\begin{pro}\label{Lpoly}
Suppose that $\pi$ and $\tau$ are unramified outside of $T$ and that $S'$ is a subset of $T$ with the property that for $x\in S'$, the local Normalized Intertwining Operator $N(s,\tau_x\otimes\Tilde{\pi}_x,\Tilde w_0)$ is holomorphic and non-zero on $\operatorname{Re}(s)\geq 1/2$ and $\Tilde w_0(\tau\otimes \Tilde\pi)\not \cong \tau\otimes \Tilde\pi$. Then the  $L$-function \[L^{T\setminus S'}(s,\pi\times \tau)=\prod_{x\not \in T\setminus S'}L(s,\pi\times \tau)\] is holomorphic on $\operatorname{Re}(s)\geq 1/2$ and non-zero on $\operatorname{Re}(s)\geq 1$.
\end{pro}
\begin{proof} In this proof we use the global intertwining operator and its relation with Eisenstein series. For the definition of the global intertwining operator we  refer to Section \ref{subsec:GlobalL}, but for their properties we rely on \cite{L18}.

Now, putting the definition of the normalized intertwining operator in the right hand side of the formula \cite[Eq. (3.2)]{L18}, we get
\begin{align*}M(s,\sigma,\Tilde w_0)f=&
\bigotimes_{x\in T\setminus
S'} A(s,\sigma_x,\Tilde w_0)f_x  \cdot \\ &\bigotimes_{x\in S'}r(s,\sigma_x)^{-1}N(s,\sigma_x,\Tilde w_0)f_x \cdot  \\ 
& \frac{L^T(s,\pi\times \tau)L^T(2s,\tau,\wedge^2\rho_m)}{L^T(1+s,\pi\times \tau)L^T(1+2s, \tau,\wedge^2\rho_m)}\bigotimes_{x\not \in T
}\Check{f}_x^\circ.
\end{align*}
Since $\Tilde w_0(\tau\otimes \Tilde\pi)\not \cong \tau\otimes \Tilde\pi$, we have that $M(s,\sigma,\Tilde w_0)$ is holomorphic on $\operatorname{Re}(s)\geq 0$ \cite[Lemma 3.3]{L18}. Using this holomorphicity result, and that   $A(s,\sigma_x,\Tilde w_0)$  and $\varepsilon$-factors are non-vanishing \cite[p. 283, Eq. (10)]{Wplan}, we have that
\[\frac{L^T(s,\pi\times \tau)L^T(2s,\tau,\wedge^2\rho_m)}{L^T(1+s,\pi\times \tau)L^T(1+2s, \tau,\wedge^2\rho_m)}\]
is holomorphic on $\operatorname{Re}(s)\geq 0$. Furthermore using that $N(s,\sigma_x,\Tilde w_0)$ is holomorphic and non-zero on $\operatorname{Re}(s)\geq 1/2$, we have that
\[\frac{L^{T\setminus S'}(s,\pi\times \tau)L^{T\setminus S'}(2s,\tau,\wedge^2\rho_m)}{L^{T\setminus S'}(1+s,\pi\times \tau)L^{T\setminus S'}(1+2s, \tau,\wedge^2\rho_m)}\]
is holomorphic on $\operatorname{Re}(s)\geq 1/2$.
From the fact that $L$-functions are holomorphic on some $\operatorname{Re} (s)>N$ \cite[Section 13.2]{CorB}, we get that 
\begin{align}
L^{T\setminus S'}(s,\pi\times \tau)L^{T\setminus S'}(2s,\tau,\wedge^2\rho_m)  \label{eq:Lpar1}    
\end{align}
is holomorphic on $\operatorname{Re}(s)\geq 1/2$. 

On the other hand, since $\Tilde w_0(\tau\otimes \Tilde\pi)\not \cong \tau\otimes \Tilde\pi$, $E(s,\Phi,g,\gb P)$ is holomorphic on $\operatorname{Re} (s) \geq 0$ \cite[Lemma 3.3]{L18}. Using this holomorphicity result and that the local $L$-functions are non-vanishing by definition and the relation \cite[Eq. (1.3)]{L18} we also get 
\begin{equation}
\begin{gathered}
\prod_{x\in |S'|}L(s,\pi_x\times \tau_x)L(1+2s, \tau_x,\wedge^2\rho_{x,m})L^{T}(1+s,\pi\times \tau)L^{T}(1+2s, \tau,\wedge^2\rho_m) \\
 =L^{T\setminus S'}(1+s,\pi\times \tau)L^{T\setminus S'}(1+2s, \tau,\wedge^2\rho_m)
  \end{gathered}
  \label{eq:Lpar2}    
 \end{equation}
is non-zero on $\operatorname{Re}(s)\geq 0$.

Now we proceed  as in \cite[Section 6.1]{L18}, to get that $L^{ T}(s,\tau,\wedge^2\rho_m)$ is holomorphic on $\operatorname{Re}(s)\geq 1/2$ and non-zero on $\operatorname{Re}(s)\geq 1$. Indeed we consider the global intertwining operator in a maximal Siegel case. Then as before, staring from \cite[Lemma 3.3]{L18}  and using \cite[Eq. (1.3) and (3.2)]{L18}, we obtain that
\[L^T(s,\tau,\wedge^2\rho_m)\]
is holomorphic on $\operatorname{Re}(s)\geq 1/2$ and non-zero on $\operatorname{Re}(s)\geq 1$. Furthermore, since every $\tau_x$ is tempered and using \cite[Corollary 5.5]{L18} (tempered $L$-function conjecture), we have 
\[L^{T\setminus S'}(s,\tau,\wedge^2\rho_m)\]
is holomorphic on $\operatorname{Re}(s)\geq 1/2$. Thus, $L^{T\setminus S'}(s,\pi\times \tau)$ is holomorphic on $\operatorname{Re}(s)\geq 1/2$ by \eqref{eq:Lpar1}. Similarly, but using \eqref{eq:Lpar2}, we get it is non-zero on $\operatorname{Re}(s)\geq 0$.\end{proof} 

 As we mention in the previous section, we already have the standard module conjecture at our disposal. Then, in order to apply Proposition \ref{pro:KimA}  and Proposition \ref{Lpoly}, we need the following local-global results. The result is inspired from \cite[Theorem 3.2]{Kim2000}.
\begin{pro} \label{pro:Lkim}
 Let $\tau$ be a (globally generic) cuspidal automorphic representation of $\gb{GL}_{m}(\mathbb{A}_F)$ and  $\pi$  a globally generic cuspidal automorphic representation of $\gb{SO}_{2n}^*(\mathbb A_F)$ such that $\tau\otimes \Tilde\pi=\sigma \not \cong \Tilde w_0\sigma$.  Then $L(s,\pi_{x_0} \times \tau_{x_0})$ is holomorphic on $\operatorname{Re}(s)\geq 1$, for every $x\in |F|$.
\end{pro}

\begin{proof}As before we input the definition of the normalized operator in the right hand side of the formula \cite[(3.2)]{L18} to get
\begin{align*}
M(s,\sigma,\Tilde w_0)f=&
\bigotimes_{x\in S\setminus\{x_0\}} A(s,\sigma_x,\Tilde w_0)f_x  \cdot \\ & r(s,\sigma_{x_0})^{-1}N(s,\sigma_{x_0},\Tilde w_0)f_{x_0} \cdot  \\ 
& \frac{L^S(s,\pi\times \tau)L^S(2s,\tau,\wedge^2\rho_m)}{L^S(1+s,\pi\times \tau)L^S(1+2s, \tau,\wedge^2\rho_m)}\bigotimes_{x\not \in S
}\Check{f}_x^\circ.
\end{align*}
Now let $N_0\geq 1$ be big enough so that $L(1+s,\pi_{x_0}\times \tau_{x_0})$ has no poles on $\operatorname{Re}(s)\geq N_0$, i.e. $L(s,\pi_{x_0}\times \tau_{x_0})$ is holomorphic on $\operatorname{Re}(s)\geq N_0+1$. This gives us that,  if $\operatorname{Re}(s)\geq N_0$, then
\[\frac{L(s,\pi_{x_0}\times \tau_{x_0})}{\varepsilon(s,\pi_{x_0}\times \tau_{x_0},\psi_{x_0})L(1+s,\pi_{x_0}\times \tau_{x_0})}\] is non-zero. Secondly, since $\tau$ is cuspidal automorphic, then thanks to \cite[Th\'eor\`eme VI.10]{L} $\tau_{x_0}$ is tempered. Then, using \cite[Corollary 5.5]{L18} we have that $L(s,\tau_{x_0},\wedge^2\rho_m)$ is holomorphic on $\operatorname{Re}(s)\geq 1$. Therefore
\[r(s,\tau_{x_0}\otimes\Tilde\pi_{x_0})=\frac{L(s,\pi_{x_0}\times \tau_{x_0})L(2s,\tau_{x_0},\wedge^2\rho_m)}{L(1+s,\pi_{x_0}\times \tau_{x_0})L(1+2s, \tau_{x_0},\wedge^2\rho_{m,x_0})\varepsilon(s,\pi_{x_0}\times \tau_{x_0},\psi_{x_0})\varepsilon(2s,\tau_{x_0},\wedge^2\rho_{m,x_0},\psi_{x_0})}\]
is non-zero on $\operatorname{Re}(s)\geq N_0$. Thirdly, using Corollary \ref{Lpoly} for $S'=\emptyset$, we get that
\[r(s,\pi_{x_0}\otimes\tau_{x_0})\frac{L^S(s,\pi\times \tau)L^S(2s,\tau,\wedge^2\rho_m)}{L^S(1+s,\pi\times \tau)L^S(1+2s, \tau,\wedge^2\rho_m)}\]
is non-zero on $\operatorname{Re}(s)\geq N_0$. 

Fourthly, recall that $M(s,\sigma,\Tilde w_0)$ is holomorphic on $\operatorname{Re}(s)\geq 0$, since $\Tilde w_0 \sigma\not \cong \sigma$ \cite[Lemma 3.3]{L18}. Thus,  using that  $A(s,\sigma_{x_0},\Tilde w_0)$ is non-zero and the equality at the beginning of the proof, we have that $N(s,\sigma_{x_0},\Tilde w_0)$ is holomorphic on $\operatorname{Re}(s)\geq N_0$.

Lastly, since the holomorphicity of $N(s,\sigma_{x_0},\Tilde w_0)$ implies its non-zeroness \cite[Theorem 3]{Zh}, we have that $N(s,\sigma_{x_0},\Tilde w_0)$ is also non-zero on $\operatorname{Re}(s)\geq N_0$. Hence
\[\frac{L(s,\pi_{x_0}\times \tau_{x_0})}{L(1+s,\pi_{x_0}\times \tau_{x_0})}\]
is holomorphic on $\operatorname{Re}(s)\geq N_0$, and thus $L(s,\pi_{x_0} \times \tau_{x_0})$ has no poles on $\operatorname{Re}(s)\geq N_0$. Arguing inductively, we conclude that $L(s,\pi_{x_0} \times \tau_{x_0})$ is holomorphic on $\operatorname{Re}(s)\geq 1$.
\end{proof}

Using the previous local-global result, we are able to prove the second property needed for the local Normalized Intertwining Operator. Again, it is inspired from \cite[Lemma 3.3]{Kim2000}.
\begin{cor}\label{cor:ram-poly}Let $\pi=\bigotimes_x'\pi_x$ be a globally generic cuspidal automorphic representation of $\gb{SO}_{2n}^*(\mathbb A_F)$. Then $r_{b,x}<1$ for every $x\in |F|$, where $r_{b,x}$ is as in equation \eqref{eq:SM}.
\end{cor}
\begin{proof}
Fix a place $x$. Let 
 \[i_P^{SO^*_{2n}}(\pi_{1,x}|\det|^{r_{1,x}}\otimes\dots\otimes\pi_{b,x}|\det|^{r_{b,x}}\otimes\pi_{0,x})\]
 be the induced representation such that $\pi_x$ is its Langlands quotient, and where $0<r_{1,x}\leq \cdots \leq r_{b,x}$, ${\bf P}$ is a parabolic subgroup of ${\bf SO}_{2n}^*$ containing $\gb B$,  $\pi_{i,x}$ is an irreducible discrete series representation of $\gb{GL}_{n_i}(F_x)$ for $1\leq i\leq b$, and $\pi_{0,x}$  is an irreducible tempered representation of $\gb{SO}_{2n_0}^*(F_x)$.

 By definition we have that $L(s-r_{b,x},\pi_{b,x}\times \Tilde{\pi}_{b,x})$ is a factor of $L(s,\pi_{x}\times \Tilde{\pi}_{b,x})$. Since $L(s-r_b,\pi_{b,x}\times \Tilde{\pi}_{b,x})$ is a Rankin-Selberg $L$-functions (Remark \ref{RSLS}), it has a pole for $s=r_{b,x}$. Now,  there is a cuspidal automorphic representation $\tau_{b}=\bigotimes_x'\tau_{b,x}$ of ${\bf GL}_{n_b}(\mathbb A_F)$, such that $\tau_{b,x}\cong \pi_{b,x}$. Using Proposition \ref{pro:Lkim}, we have that $L(s,\pi_{x}\times \Tilde{\pi}_{b,x})$ is holomorphic on $\operatorname{Re}(s)>1$, and thus we must have that $r_{b,x}<1$.
\end{proof}

Now, we are finally ready to prove the polynomial condition.  First, we can find as in \cite[Proposition 4.1]{L18} a sufficiently ramified character $\eta_{x_0}$, with $x_0\in S$ (nonempty by definition), such that $\Tilde w_0 (\pi\otimes \tau \cdot \eta )\not \cong \pi\otimes (\tau\cdot \eta)$. 
This allows us to obtain the condition needed for Proposition \ref{pro:Lkim}. Furthermore, combining Corollary \ref{cor:ram-poly} and Proposition \ref{pro:KimA}, we obtain that the local normalized intertwining operator
\[N(s, \tau_{x_0}\otimes \Tilde \pi_{x_0},\Tilde w_0)\]
is holomorphic and non-zero on $\operatorname{Re}(s)\geq 1/2$, for every inert place  $x\in |F|$. Now, using an analogous result of Proposition \ref{pro:KimA} in the split case $\gb{SO}_{2n}$ in \cite{dCHL} and applying  Proposition \ref{Lpoly} to the set $T=S'$,  we obtain
that 
\[L(s,\pi\times\tau)=\prod_{x \in |F|}L(s,\pi_x \times\tau_x)\]
is holomorphic on $\operatorname{Re}(s)\geq 1/2$. Finally using the Langlands-Shahidi functional equation \eqref{eq:funeqLS}, we get that $L(s,\pi\times\tau )$ is entire. In addition, using the rationality property of $L$-functions \cite[Theorem 1.2]{L18} we see that $L(s,\Pi\times \tau)$ is a polynomial.

(\emph{Trivial on $F^\times$}). Choosing characters $\nu_x$ for $x\in S$ sufficiently ramified as in the polynomial condition and  in the ramified case of  \eqref{ram}, we can find a character $\eta$ of $\mathbb{A}^\times_F$ trivial on $F^\times$ \cite[X, Theorem 5]{ATCFT} and which satisfies $\eta_x=\nu_x$ for $x\in S$.

\label{transfer} As we have checked all the hypothesis of the converse Theorem \ref{Con} in the previous section, we find an irreducible automorphic representation $\Pi'$ of $\gb{GL}_{2n}(\mathbb A_F)$ such that $\Pi_x'=\Pi_x$ for $x\notin S$. 
\begin{thm}\label{thm:weak}
Let $\pi = \bigotimes_x' \pi_x$ be a globally generic cuspidal automorphic representation of ${\bf SO}_{2n}^*(\mathbb{A}_F)$, unramifie outside of a finite set of places $S$. Then there exists an automorphic representation $\Pi = \bigotimes_x'\Pi_x$ of ${\bf  GL}_{2n}(\mathbb{A}_F)$  such that  the representation $\Pi_x$ is unramified for every $x\not \in S$ and that its Satake parameter $\phi_{\Pi_x}$ satisfy that $\phi_{\Pi_x}=\rho_{2n,x}^*\circ \phi_{\pi_x}$  and where its central character is given by \eqref{central}.
\end{thm}

We will study further properties of these lifts, in the next section.

\section{Automorphic L-functions and image of the functorialiy}

 Inspired by \cite{L18} and \cite{Sdry}, we prove that the cuspidal factors of the isobaric sum are distinct, unitary and self dual, in positive characteristic.  
  We check that this lift respects the arithmetic information coming from  $\gamma$-factors. We finish by proving, as an application of the functoriality, the unramified Ramanujan conjecture for globally generic cuspidal automorphic representations of $\gb{SO}_{2n}^*(\mathbb A_F)$.

  To obtain the result about the isobaric sum, we need the following fact about the unramified unitary dual of the split special orthogonal group ${\bf SO}_{2n}$ \cite{dCHL}.

\begin{pro}\label{A:non-uni}
Let $\tau$ be a tempered smooth irreducible unramified representation of $\gb{GL}_m(F)$ and $\pi$ a unitary smooth irreducible unramified representation of $\gb{SO}_{2n}(F)$. Let s be a complex parameter with $\operatorname{Re}(s)>1$. Then the unramified component of $i_{{\bf P}(F)}^{{\bf SO}_{2(n+m)}(F)}(|\det|^s\tau \otimes \pi)$ is not unitary.
\end{pro}
\subsection{Global L-functions} \label{subsec:GlobalL}With the unramified unitary dual result at hand, we are ready to continue with the study of the image of functoriality.

Let $\gb P=\gb M\gb N$ be a maximal parabolic subgroup of $\gb G$ containing $\gb B$ associated to $\theta=\Delta- \{\alpha\}\subset \Delta$  and let $w_0=w_{l,\gb G}w_{l,\gb M}\in W^{\gb G}$, where $w_{l,{\bf G}}$ and $w_{l,{\bf M}}$ are the longest element of the Weyl group of ${\bf G}$ and of ${\bf M}$, respectively. We fix a maximal compact open subgroup $K=\prod_x K_x$ of $G=\gb G(\mathbb A_F)$, as in \cite[Section I.1.4]{CW}.

Let $\sigma=\bigotimes_x'\sigma_x$ be a unitary cuspidal automorphic representation of $\gb M(\mathbb A_F)$, where the restricted product is taken with respect to functions \{$\varphi^\circ_{x}$\}. We write
\[i_P^G(s,\sigma)=\bigotimes_x{}'  i_{{\bf P}(F_x)}^{{\bf G}(F_x)}(s,\sigma_x)=\bigotimes_x{}' i_{{\bf P}(F_x)}^{{\bf G}(F_x)} (\sigma_x \otimes q^{\langle s\Tilde \alpha,H_{\gb P_x}(\cdot)\rangle}),\]
where the restricted product is taken with respect to the functions $f^\circ_{x,s}\in i_P^G(s,\sigma_x)$ such that $f^\circ_{x,s}(k_x)=\varphi_x^\circ$ for all $k_x\in K_x$.
For $\Tilde w_0$ a representative of $w_0$, we define the global intertwining operator for $\operatorname{Re}(s)$ big enough, as in \cite[Section 1.2]{L18}, 
by
\begin{align*}
M(s ,\sigma,\Tilde w_0)\colon i_P^G(s,\sigma)&\to i_{P'}^G(\Tilde w_0(s),\Tilde w_0(\sigma))\\
f&\mapsto\left(g\mapsto\int_{\gb N'(\mathbb A_F)} f(\Tilde w^{-1}_0ng)dn\right),
\end{align*}
 where $\gb N'$ is the radical of $\gb P'=\gb P_{w_0(\theta)}$.


\begin{pro}\label{Mholo2}Suppose that $\gb G=\gb{SO}^*_{2(m+n)}$, let $\gb P=\gb M\gb N$ be a parabolic subgroup containing $\gb B$ with Levi subgroup $\gb M$ isomorphic to   $\gb{GL}_m \times \gb{SO}^*_{2n}$ and $\Tilde w_0\in \gb G(F)$  a representative of $ w_0\in W^{\gb G}$. Let $\sigma=  \tau \otimes  \Tilde\pi $ be a unitary globally generic cuspidal automorphic representation of $\gb M(\mathbb A_F)$. Then $M(s,\sigma,\Tilde w_0)$ is holomorphic on $\operatorname{Re}(s)>1$.
\end{pro}
\begin{proof} 
Let $S$ be a finite subset  of $|F|$, such that $\sigma_x$ is unramified for $x\not \in S$.  

Thanks to the work of L. Lafforgue \cite[Th\'eor\`eme VI.10]{L}, we know that each local component of the globally generic cuspidal automorphic representation $\tau=\bigotimes_x'\tau_x$ of $\gb{GL}_m(\mathbb A_F)$ 
is tempered. Then, for each $x\not \in S$ we have that $\tau_x$ is of the form
\[i_{B_m}^{GL_m}(\chi_{1,x}\otimes \cdots\otimes \chi_{m,x}),\]
where $\chi_{j,x}$ is unitary unramified character of $F^\times_x$.

 Now, if $s_0$ is a pole of $M(s,\sigma,\Tilde w_0)$, then some subquotient of $i_{P}^{G} (s_0,\sigma)$  would be in the discrete residual spectrum   \cite[Section 1]{Kim2000}, thus unitary. Then for such $s_0$, we would have that for almost every $x\in |F|$,  the unramified component of  $i_{\gb P(F_x)}^{\gb G(F_{x})}(s_0,\sigma_x)$ is unitary. 
 
 We argue by contradiction, and thus we assume that $\operatorname{Re}(s_0)>1$. First, we can enlarge $S$, so that $i_{\gb P(F_x)}^{\gb G(F_{x})}(s_0,\sigma_x)$  has a unitary unramified component. As the set $\{x\in |F|\setminus S \colon x \text{ is split  in } E\}$ has density $1/2$, by Chebotarev's theorem, we can always some find $x_0\not \in S$ split  in $E$.  But, thanks to Proposition \ref{A:non-uni}, we get that the unramified component of $i_{\gb P(F_x)}^{\gb G(F_{x})}(s_0,\sigma_x)$ is not unitary, thus a contradiction.
\end{proof}

\begin{thm} \label{GLS1}Suppose that $\gb G=\gb{SO}^*_{2(m+n)}$, and $\gb P=\gb M\gb N$ parabolic subgroup with Levi subgroup $\gb M$ isomorphic to   $\gb{GL}_m \times \gb{SO}^*_{2n}$. Let $\sigma=\tau \otimes \Tilde\pi$ be a globally generic cuspidal automorphic representation of $\gb{M}(\mathbb A_F)$. Then $L^S(s,\pi \times \tau)$ is holomorphic and non-vanishing on $\operatorname{Re} (s)>1$ and has at most a simple pole at $s=1$.
\end{thm}
\begin{proof}
Let $S$ be a finite subset  of $|F|$, such that $\sigma_x$ is unramified for $x\not \in S$, as in the proof of Proposition \ref{Mholo2}. From \cite[Eq. (3.2)]{L18} and Proposition \ref{Mholo2}, we have that 
\[\frac{L^S(s,\pi\times \tau)L^S(2s,\tau,\wedge^2\rho_m)}{L^S(1+s,\pi\times \tau)L^S(1+2s, \tau,\wedge^2\rho_m)}\]
is holomorphic on $\operatorname{Re} (s)>1$. As $L^S(s,\tau,\wedge^2\rho_m)$ is holomorphic and non-zero on $\operatorname{Re}(s)>1$ \cite[Corollary 6.4]{L18}, we can conclude that 
\[\frac{L^S(s,\pi\times \tau)}{L^S(1+s,\pi\times \tau)}\]
is holomorphic on $\operatorname{Re} (s)>1$.

On the other hand, as  \cite[Proposition II.1.7]{CW}
\[E_{\gb P}(s,f,g,\gb P)=\int_{\gb U(F)\backslash \gb U(\mathbb A_F)}E(s,f,ug,\gb P)du=f(g)+M(s,\sigma,\Tilde w_0)f(g)\] and  Proposition \ref{Mholo2}, we have that the poles  of the constant terms $E_{\gb P}(s,f,g,\gb P)$ are contained in $\operatorname{Re}(s)\leq 1$.  Since $\gb U(F)\backslash \gb U(\mathbb A_F)$ is compact, the formula \cite[Eq. (1.3)]{L18} and that $L^S(s,\tau,\wedge^2\rho_m)$ is holomorphic and non-zero on $\operatorname{Re}(s)>1$, we conclude that $L^S(1+s,\pi\times \tau)^{-1}$
is holomorphic and non-vanishing on $\operatorname{Re}(s)>1$. Thus $L^S(s,\pi\times \tau)$ is holomorphic on $\operatorname{Re} (s)>1$. 
 
 Finally, we have that the poles of the global intertwining operator are all simple \cite[Proposition IV.1.11, (c)]{CW}. Then,  again by \cite[Eq. (3.2)]{L18} and the non-zeroness of the local intertwining operators $A(s,\tau_x\otimes \Tilde\pi_x ,\Tilde{w}_0)$, the quotient
\[\frac{L^S(s,\pi\times \tau)L^S(2s,\tau,\wedge^2\rho_m)}{L^S(1+s,\pi\times \tau)L^S(1+2s, \tau,\wedge^2\rho_m)}\]
 has at most a simple pole at $s=1$. From \cite[Corollary 6.4]{L18}, $L^S(2,\tau,\wedge^2\rho_m)$ and $L^S(3,\Tilde{\tau},\wedge^2\rho_m)$ are non-zero. Thus, $L^S(s,\pi\times \tau)$ has at most a simple pole at $s=1$.
\end{proof}

\subsection{Image and isobaric sums}
We recall that from \cite[Proposition 2]{RLiso}, every automorphic representation $\Pi$ of $\gb{GL}_{2n}(\mathbb A_F)$ arises as a subquotient of a representation induced from cuspidal automorphic  representations,
\begin{equation}
     i^{\gb{GL}_{2n}(\mathbb A_F)}_{\gb P(\mathbb A_F)}(\Pi_1\otimes\cdots\otimes\Pi_d) \label{indglob},\end{equation}
 where $\gb P$ is a parabolic subgroup of $\gb{GL}_{2n}$ containing the Borel subgroup of $\gb{GL}_{2n}$ consisting of upper triangular matrices, and with every $\Pi_i$ a cuspidal  automorphic representation of $\gb{GL}_{n_i}(\mathbb A_F)$.  We also recall that Langlands isobaric sum  gives us a subquotient of \eqref{indglob}, that we denote \cite[Section 2]{RLbaric}
 \[\Pi_1\boxplus \cdots \boxplus \Pi_d.\]
 
 We are going to study the representation $\Pi_j$ for every $1\leq j\leq d$ obtained from the representations in the image of functiorality.
\begin{thm}\label{thm:image} Let $\pi$ be a globally  generic cuspidal automorphic  representation of $\gb{SO}_{2n}^*(\mathbb A_F)$. Then, $\pi$ transfers to a globally generic irreducible automorphic representation $\Pi$ of $\gb{GL}_{2n}(\mathbb A_F)$. Its central character is given by  \eqref{central} and $\Pi$ can be expressed as an isobaric sum
\[\Pi=\Pi_1 \boxplus \cdots  \boxplus\Pi_d,\]
where each $\Pi_i$ is a unitary self-dual cuspidal automorphic representation of $\gb{GL}_{N_i}(\mathbb A_F)$, and $\Pi_i\not\cong \Pi_j$ for $i\neq j$.
\end{thm}
\begin{proof} 

The existence of the transfer of $\pi$  and its central character is Theorem \ref{thm:weak}. We now show the properties of $\Pi_i$. Let $S$ be a finite set of $|F|$ such that $\pi$ is unramified outside of $S$.

(\textit{Unitarity}). We  write $\Pi_i=|\det|^{n_i}\Pi_i'$, where $\Pi_i'$ is unitary for  every $1\leq i \leq d$ and $n_d\geq\cdots\geq n_1$. Given that the central character of $\Pi$ is unitary, we have that  $n_1\leq 0$.
By  \eqref{unr}  and the multiplicativity property of Rankin-Selberg $L$-functions we have
\begin{align*}
L^S(s,\pi \times \Tilde \Pi_1)=L^S(s,\Pi\times \Tilde{\Pi}_1')&=\prod_j L^S(s,\Pi_j\times \Tilde{\Pi}_1')\\
&=\prod_jL^S(s+n_j,\Pi_j'\times\Tilde{\Pi}_1').
\end{align*}
Since the left hand side has at most a pole at $s=1$ and it is holomorphic and non-vanishing for $\operatorname{Re}(s)>1$ by Theorem \ref{GLS1}, we must have that $n_1=0$. Recursively we can check that $n_i=0$ for all $i$. Thus $\Pi_i$ is unitary for all $i$.

As a consequence we have that $\Pi$ is equal to the isobaric sum of the $\Pi$'s, as each $\Pi_i$ is unitary and thus $\Pi$ is the full induced representation. Moreover, as the $\Pi_i$'s are globally generic then $\Pi$ is also globally generic.

(\textit{Distinct}). As before we consider
\begin{align*}
L^S(s,\pi\times\Tilde \Pi_i )=L^S(s,\Pi\times \Tilde{\Pi}_i)&=\prod_j L^S(s,\Pi_j\times \Tilde{\Pi}_i)\\
&=\prod_jL^S(s,\Pi_j\times\Tilde{\Pi}_i)
\end{align*}
Arguing as above, we must have $\Pi_i\not \cong \Pi_j$ for $i\neq j$, because otherwise the right hand side would not have a simple pole \cite[II;  (3.6)] {RS}. 

(\textit{Self-dual}).  First observe that linear map $\Tilde w_0$  of $Q_{F,n+m}$ (See Section \ref{subsec:SO*}), given by $\Tilde w_0e_i=e_{2(n+m)-(n-i)}$ for $1\leq i \leq n$, $\Tilde w_0e_i=e_i$ for $n+1\leq i \leq 2m+n$, trivial on $l$ and $\Tilde w_0e_i=e_{i-n-2m}$ for $n+2m+1\leq i\leq 2n+2m$ is in $\gb{SO}^*_{2n+2m}(F)$ and is a representative of $w_0=w_{l,\gb G}w_{l,\gb M}\in W^{\gb G}$, where $\gb M\cong \gb{GL}_m\times \gb{SO}_{2n}^*$. The action of $\Tilde w_0$ on $(g_1,g_2)\in \gb M(\mathbb A_F)$ is $({}^tg^{-1}_1,g_2)$. Furthermore, $\Tilde w_0(\sigma)=\Tilde{\Pi_i}\otimes  \Tilde \pi$. Assume that $\Pi_i$ not selfdual. In that we would have $\sigma\not \cong \Tilde w_0(\sigma)$. In that case, Corollary \ref{Lpoly} implies that  the left hand side

\begin{align*}
L^S(s,\pi  \times \Tilde \Pi_i)=L^S(s,\Pi\times \Tilde{\Pi}_i)&=\prod_j L^S(s,\Pi_j\times \Tilde{\Pi}_i)\\
&=\prod_jL^S(s,\Pi_j\times\Tilde{\Pi}_i)
\end{align*}
is holomorphic on $\operatorname{Re}(s)>1/2$. But the right hand side has a pole coming from $L(s,\Pi_i\times\Tilde{\Pi}_i)$ \cite[II; (3.6)]{RS}. A contradiction, thus the $\Pi_i$'s are self-dual.
\end{proof}

\begin{rmk}
We can compare this construction with the following construction that uses a combination of recent results of V. Lafforgue and the already used result of L. Lafforgue \cite{L,VL}. More precisely, given a cuspidal automorphic representation $\pi$ of ${\bf SO}_{2n}^*(\mathbb A_F)$, we get a  Langlands parameter $\phi: \Gamma_F \to {}^LG(\overline{\mathbb Q}_\ell)$ of ${\bf SO}_{2n}^*$. After composing with $\rho_{2n}^*$ (the $\ell$-adic version), we get a (semisimple) $2n$-dimensional Galois representation $V$ \cite{VL}. We let $V_l$ for $1\leq l \leq e$ be its irreducible subrepresentations. Every irreducible representation $V_l$  corresponds to an irreducible cuspidal automorphic representation $\Pi_{L,i}$ of ${\bf  GL}_{M_l}(\mathbb A_F)$ \cite[Th\'eor\`eme VI.9]{L}. We consider the parabolically induced representation 
\[i_{{\bf Q}(\mathbb A_F)}^{{\bf GL}_{2n}(\mathbb A_F)}(\Pi_{L,1}\otimes \cdots \otimes \Pi_{L,e}).\]
Every irreducible subquotient of this representation is a lift.  

If $\pi$ is globally generic, we will compare the construction above with the construction $\Pi=\Pi_1 \boxplus \cdots  \boxplus\Pi_d$, from Theorem \ref{thm:image}.  Since the representations 
\[i_{{\bf Q}(\mathbb A_F)}^{{\bf GL}_{2n}(\mathbb A_F)}(\Pi_{L,1}\otimes \cdots \otimes \Pi_{L,e}) \quad  \& \quad i_{{\bf P}(\mathbb A_F)}^{{\bf GL}_{2n}(\mathbb A_F)}(\Pi_{1}\otimes \cdots \otimes \Pi_{d})  \] have the same unramified component, we have that $d=e$ and that there exists a bijection $\sigma$ of $\{1,\dots,e\}$ such that $\Pi_i\cong \Pi_{\sigma(i),L}$ \cite[Theorem 4.4]{RS}. Therefore, as $\Pi_i$ is also unitary for every $1\leq i \leq d$ and thus the induced representation above is irreducible, we have that both constructions coincide. Finally, as the Langlands parameter is orthogonal, we obtain that the representation $\Pi$ is orthogonal.
\end{rmk}

We finish this section with a result about the compatibility between the $\gamma$-factors of the representation $\pi$ and its transfer $\Pi$.
\begin{thm} \label{gamlocal} Let $\Pi_x$ be the local component of a transfer as in Theorem \ref{thm:image} and  $\tau_x$  be an irreducible  generic   representation of $\gb{GL}_m(F_x)$. Then
\[\gamma(s, \pi_x\times \tau_x,\psi_x) =\gamma(s,\Pi_{x} \times \tau_x,\psi_x)\label{gammalocal}.\]
\end{thm}
\begin{proof}
We first note that this is true when $\pi_x$ is unramified, i.e when $x\not \in S$ \eqref{unr}.

Let us fix $x_0\in |F|$  and suppose first that $\tau_{x_0}$ is cuspidal.  Then there is a cuspidal automorphic representation $\tau=\bigotimes' \tau_x$ of $\gb{GL}_m(\mathbb A_F)$ that is $\tau_{x_0}$ at $x_0$ and such that $\tau_x$ is unramified for $x\not \in S$ \cite[Lemma 3.1]{L18}. Furthermore, thanks to the Grunwald-Wang theorem \cite[X, Theorem 5]{ATCFT} and Section \ref{subsec:locallift} (central character of $\Pi_x$ is $\varkappa_x$), we can choose a character $\eta=\otimes \eta_x$ such that $\eta_x$ is sufficiently ramified for $x\in S$ and $x\neq x_0$ so that 
\[\gamma(s, \pi_x\times (\tau_x\cdot\eta_x),\psi_x)=\gamma(s,\Pi_x\times (\tau_x\cdot \eta_x),\psi_x)\]
and $\eta_{x_0}=1$.

On the other hand, the Langlands-Shahidi  functional equation \cite[Theorem 5.1 (vi)]{L15}  gives us that
\begin{align*}
L^S(s, \pi\times (\tau\cdot\eta))=\gamma(s&, \pi_{x_0}\times (\tau_{x_0}\cdot\eta_{x_0}),\psi_{x_0}) \\
&\prod_{x\in S-\{x_0\}}\gamma(s,\pi_x\times (\tau_x\cdot\eta_x),\psi_x)  L^S(1-s,\Tilde \pi\times (\Tilde\tau\cdot\Tilde \eta)).    
\end{align*}
Similarly for the Rankin-Selberg $L$-functions
\[L^S(s,\Pi \times (\tau\cdot \eta))=\gamma(s,\Pi_{x_0} \times \tau_{x_0} ,\psi_{x_0})  \prod_{x\in S-\{x_0\}}\gamma(s,\Pi_x \times (\tau_x\cdot \eta_x),\psi_x)  L^S(1-s,\Tilde \Pi \times (\widetilde{\tau}\cdot \Tilde\eta)).\]
Thus, after simplifying we get 
\[\gamma(s, \pi_{x_0}\times \tau_{x_0},\psi_{x_0}) =\gamma(s,\Pi_{x_0} \times \tau_{x_0},\psi_{x_0}),\]
obtaining thus the relation for the cuspidal representation $\tau_x$. 

The representation $\tau_x$ can be expressed as subquotient  of 
    \[i_{\gb P(F_x)}^{\gb{GL}_m(F_x)}(\rho_{1}\otimes \cdots  \otimes
\rho_d), \]
where $\gb P$ is a parabolic subgroup containing the Borel subgroup of $\gb{GL}_m$ consisting of upper triangular matrices and the $\rho_i$'s are generic cuspidal representations of $\gb{GL}_{m_i}(F_x)$. Using the multiplicativity property of $\gamma$-factors \eqref{eq:mult1} we get
\begin{align*}
\gamma(s, \pi_{x}\times \tau_{x},\psi_x)=&
\prod_{i=1}^d \gamma(s, \pi_x \times \rho_{i} ,\psi_x).
\end{align*}
 Similarly,
  \begin{align*}
\gamma(s,\Pi_x\times \tau_x,\psi_x)=&
\prod_{i=1}^d \gamma(s,\Pi_x\times\rho_i,\psi_x).
\end{align*}
 As we know the desired relation when the representation is cuspidal, we obtain
\begin{align*}
\gamma(s, \pi_x\times \tau_x,\psi_x)&=\prod_{i=1}^{d} \gamma(s, \pi_x \times \rho_i  ,\psi_x)\\
&=\prod_{k=1}^{d} \gamma(s,\Pi_x\times \rho_i,\psi_x)\\
&=\gamma(s,\Pi_x\times \tau_x,\psi_x).
\end{align*}
\end{proof}
\begin{rmk} \label{rmk:Lift}\begin{itemize}
    \item 
 First, we remark that $\Pi$ only depends on $\pi$, thanks to the multiplicity one result for isobaric sums \cite{RS}. We also note that, combining Theorem \ref{gammalocal} and \cite[Th\'eor\`eme 1.1]{GHcar}, we can prove that $\Pi_x$ only depends on $\pi_x$. Finally, we also mention that, conjecturally, the image  is characterised by the condition in the Theorem \ref{thm:image} and the fact that $L^T(s,\Pi_i,\operatorname{Sym}^2)$ has a pole at $s=1$ for any sufficiently large finite set  of  places $T$ containing all archimedean places. This is established in the work of Cogdell, Piatetski-Shapiro and  Shahidi over number fields \cite{Cogqsplit}. 
\item We expect a relation between $L$-functions and $\varepsilon$-factors, that is similar to the one between $\gamma$-factors in Theorem \ref{gammalocal}. In positive characteristic, we have these relations established for the split classical groups \cite{L09}. 
\end{itemize}
\end{rmk}

 
\subsection{Unramified Ramanujan conjecture}
\begin{thm}\label{rama}
 Let $\pi=\otimes_x \pi_x$ be a globally generic cuspidal representation of $\gb{SO}_{2n}^*(\mathbb A_F)$.  If $\pi_x$ is unramified, then its Satake parameter has absolute value 1.
\end{thm}
\begin{proof}
Let us fix $x\in |F|$ an inert place. The split case is obtained as in \cite[Theorem 9.14]{L09}. Using Theorem \ref{thm:image}, we have can find a transfer $\Pi$ of $\pi$ that is the isobaric sum
\[\Pi=\Pi_1 \boxplus \cdots  \boxplus\Pi_e,\]
where each $\Pi_i$ is a unitary self-dual cuspidal automorphic representation of $\gb{GL}_{N_i}(\mathbb A_F)$, and $\Pi_i\not\cong \Pi_j$ for $i\neq j$. By \cite[Th\'eor\`eme VI.10]{L}, each $\Pi_{i,x}$ is tempered.

If $\pi_x$ is unramified, let 
\[\operatorname{diag}(\alpha_1,\cdots,\alpha_{n-1},1)\rtimes \operatorname{Fr}_x\]
be its semisimple conjugacy class.
Then, by definition, the semisimple conjugacy class of $\Pi_x$ is given by
\[\operatorname{diag}(\alpha_1,\cdots,\alpha_{n-1},1,1,\alpha_{n-1}^{-1},\cdots,\alpha_{1}^{-1}).\]
 Every  $\alpha_j$ or $\alpha_j^{-1}$ is the Satake parameter of one the representations $\Pi_{i,x}$. As every $\Pi_{i,x}$ is tempered, we must have
 \[|\alpha_i|=1.\]

\begin{rmk}
If one proves the relation between $L$-functions, mentioned in Remark \ref{rmk:Lift}, then we expect to prove that  $\pi_x$ is tempered for every $x\in |F|$.
\end{rmk}


\end{proof}

\end{document}